\documentclass[a4paper,12pt]{article}
\usepackage[hmargin=2.5cm,vmargin=1.5cm]{geometry}
\usepackage{amsfonts}
\usepackage{amssymb}
\usepackage{amsmath}
\usepackage{amsthm}
\usepackage{amscd}
\usepackage{geometry}
\usepackage{array}
\usepackage{authblk}
\usepackage{float}
\usepackage{slashbox}
\usepackage{mathtools}
\usepackage{pgfplots}
\usepackage{etex}
\usepackage{pst-node}
\usepackage{tikz-cd}
\usepackage{tikz-qtree}

\usepackage{mathtools}
\DeclarePairedDelimiter\ceil{\lceil}{\rceil}
\DeclarePairedDelimiter\floor{\lfloor}{\rfloor}

\newtheorem{definition}{Definition}[section]
\newtheorem{theorem}{Theorem}[section]
\newtheorem{lemma}{Lemma}[section]
\newtheorem{proposition}{Proposition}[section]
\newtheorem{remark}{Remark}[section]
\newtheorem{corollary}{Corollary}[section]
\newtheorem{example}{Example}[section]

\newcommand{\om}{\textnormal{\textbf{o}}}

\newcommand{\som}[2]{\overset{#2}{\underset{i=#1} {\textnormal{\large\textbf{O}}}    }}

\newcommand{\somj}[2]{\overset{#2}{\underset{j=#1} {\textnormal{\large\textbf{O}}}    }}

\newcommand {\im}{ \mathop {\rm Im} \nolimits }

\begin{document}

\title{Fixed points of generalized contractions on O-metric spaces}

\date{}

\author[1]{\small Hallowed Oluwadara Olaoluwa}
\author[2]{\small Aminat Olawunmi Ige}
\author[3]{\small Johnson Olajire Olaleru}

\affil[1]{\small Corresponding author}

\affil[1]{\small email address:  holaoluwa@unilag.edu.ng}
\affil[2]{\small email address: aminat.ige@lasu.edu.ng}

\affil[3]{\small email address: jolaleru@unilag.edu.ng}

\affil[1,3]{\small Department of Mathematics, University of Lagos, Nigeria}
\affil[2]{\small Department of Mathematics, Lagos State University, Nigeria}

\maketitle

\begin{abstract}
\noindent
	The class of O-metric spaces generalize several existing metric-types in literature including metric spaces, b-metric spaces, and ultra metric spaces. 
In this paper, we discuss the properties of the topology induced by an O-metric and establish in the setting, fixed point theorems of contractions and generalized contractions. The proofs of the theorems
rely heavily on polygon $\om$-inequalities which are a natural generalization of the triangle inequality, and the construction of which leads to the notion of $\om$-series following a pattern of functions.
\vspace{5mm}

\vspace{5mm}
	
	
\end{abstract}


\section{Introduction}
The class of O-metric spaces was recently introduced in \cite{ourpaper} as a robust unification and generalization of metric spaces and metric-types such b-metric spaces \cite{Bakhtin1989,Czerwik1993}, ultra metric spaces \cite{ultra}, multiplicative metric spaces \cite{mult}, b-multiplicative metric spaces \cite{mgraph}, $\theta$-metric spaces \cite{thetamet}, $p$-metric spaces \cite{pbest}, and several extended b-metric spaces \cite{include1} and controlled metric type spaces \cite{discover}. In O-metric spaces, self-distance is set at a nonnegative real number, and the triangle inequality is amended to accommodate many binary operations on the set $\mathbb{R}$ of real numbers. In this section, we present some key definitions and results presented in \cite{ourpaper}.
\\
\\
Throughout this article, $\mathbb{N}$ denotes the set $\{1,2,3,\ldots\}$ of natural numbers, $\mathbb{N}_0$ is the set of non-negative integers, $\mathbb{R}$ denotes the set of all real numbers, $\mathbb{R}_+$ denotes the interval $[0,\infty)$ of real
numbers, $a$ is a non-negative real number,  $I_a$ is an interval in $\mathbb{R}_+$ containing $a$ and  $\mathnormal{\mathbf{o}} :I_a \times I_a \to \mathbb{R}_+$ is a function whose values $\mathnormal{\mathbf{o}}(u,v)$ are also denoted $u~ \mathnormal{\mathbf{o}}~v$, where $u,v \in I_a$. Moreover, $X$ denotes a non-empty set, and $d_{\mathnormal{\mathbf{o}}}$ a function defined for all pairs of elements of $X$, whose values are in the interval $I_a$. The floor function and the ceiling function are denoted by $\floor*{.}$ and $\ceil*{.}$ respectively: for any $x \in \mathbb{R}$, $\floor*{x}:=\sup\{n \in \mathbb{Z}: ~ n \leq x\}$ and $\ceil*{x}:=\inf\{n \in \mathbb{Z}: ~ n \geq x\}$, where $\mathbb{Z}$ is the set of all integers.


\begin{definition}
\label{psimetric} Let $X$ be a nonempty set. A function $d_\mathnormal{\mathbf{o}}: X \times X \to I_a$ is said to be an O-metric (or more specifically, an $\mathnormal{\mathbf{o}}$-metric) on $X$ and $(X,d_\mathnormal{\mathbf{o}},a)$ an O-metric space (or $\mathnormal{\mathbf{o}}$-metric space), if for all $x,y,z \in X$, the following conditions hold:
\begin{itemize}
\item[(i)] $d_\mathnormal{\mathbf{o}}(x,y)=a$ if and only if $x=y$;
\item[(ii)] $d_\mathnormal{\mathbf{o}}(x,y)=d_\mathnormal{\mathbf{o}}(y,x)$;
\item[(iii)] $d_\mathnormal{\mathbf{o}}(x,z) \leq d_\mathnormal{\mathbf{o}}(x,y) ~\mathnormal{\mathbf{o}} ~d_%
\mathnormal{\mathbf{o}}(y,z)$ (triangle $\mathnormal{\mathbf{o}}$-inequality). 
\end{itemize}
In particular, if $I_a \subset [a,\infty)$, $X$ is said to be an $a$-upward (or upward) O-metric space, and if $I_a \subset
[0,a]$, $X$ is said to be an $a$-downward (or downward) O-metric space.\footnote{The term ``$a$-upward O-metric" (``$a$-downward O-metric") refers to O-metrics for which the values are greater (less) than or equal to $a$, the value at pairs with same elements.}
\end{definition}
\noindent
Many metric types in literature are $0$-upward or $1$-upward O-metric spaces:

\begin{example}\label{referencemet}
\textup{
An upward O-metric space $(X,d_\mathnormal{\mathbf{o}},a)$ is:
\begin{itemize}
\item[(i)] a metric space $(X,d)$ when $a=0$, $\om(u,v)=u+v$ for all $u,v \geq 0$, and $d=d_{\mathnormal{\mathbf{o}}}$;
\item[(ii)] a b-metric space $(X,d,s)$ when $a=0$, $\mathnormal{\mathbf{o}}(u,v)=s(u+v)$ for all $u,v \geq 0$, with $s \geq 1$, and $d=d_{\mathnormal{\mathbf{o}}}$;
\item[(iii)] a multiplicative metric space $(X,d_\times)$ when $a=1$, $\mathnormal{\mathbf{o}}(u,v)=uv$ for all $u,v \geq 1$, and $d_\times=d_{\mathnormal{\mathbf{o}}}$;
\item[(iv)] a b-multiplicative metric space $(X,d_\times,s)$ when $a=1$, $\mathnormal{\mathbf{o}}(u,v)=(uv)^s$ for all $u,v \geq 1$ and for some $s \geq 1$, and $d_\times=d_{\mathnormal{\mathbf{o}}}$; 
\item[(v)] an ultra metric space $(X,d_\wedge)$ when $a=0$, and $\om(u,v)=\max\{u,v\}$ for all $u,v \geq 0$, and $d_\wedge=d_{\mathnormal{\mathbf{o}}}$;
\item[(vi)] a $\theta$-metric space when $a=0$, and $\mathnormal{\mathbf{o}}$ is a symmetric and continuous function $\theta$, strictly increasing in each variable, with $\theta(0,0) = 0$ and $\theta(s,0) \leq s$ for all $s>0$, and for each $r \in \im(\theta)$ and $s \in [0,r],$ there exists $t \in [0,r]$ such that $\theta(t,s) = r$;
\item[(vi)] a p-metric space $(X,\tilde{d})$ if $a=0$, $\tilde{d}=d_{\mathnormal{\mathbf{o}}}$ and $\mathnormal{\mathbf{o}}(u,v)=\Omega(u+v)$ for all $u,v \geq 0$, where $\Omega:[0,\infty) \to [0,\infty)$ is a strictly increasing continuous function with $t \leq \Omega(t)$ for all $t \geq 0$.
\end{itemize}
}
\end{example}

\noindent The following are examples of O-metrics which are new in the existing 
literature on distance spaces:

\begin{example}
\label{example2} 
\textup{
The set of real numbers $\mathbb{R}$ with $d_\mathnormal{\mathbf{o}}(x,y)=\ln(1+|x-y|)$ for all $x,y \in X$, is a $0$-upward $\mathnormal{\mathbf{o}}$-metric space
where $\mathnormal{\mathbf{o}}(u,v)= (u+1)(v+1)$ for all $u,v \in [0, \infty)$. 
}
\end{example}

\begin{example} 
\label{genl}
\textup{
Define $g(t)=a+|t-a|$ for $t \in \mathbb{R}$. Let $a \in [0,\infty)$ be a real number, $I_a=[a,\infty)$, and let $\mathnormal{\mathbf{o}}$ be increasing in each of its variables with  $\mathnormal{\mathbf{o}}(a,a)=a$, $u \leq \mathnormal{\mathbf{o}}(u,a)$ and $\mathnormal{\mathbf{o}}(u,v)=\mathnormal{\mathbf{o}}(v,u)$ for all $u,v \in I_a$. Then $(\mathbb{R},d_{\mathnormal{\mathbf{o}}}, a)$ is an $\mathnormal{\mathbf{o}}$-metric space, where for any $x,y \in \mathbb{R}$, $d_{\mathnormal{\mathbf{o}}}(x,y)=g(x) ~\mathnormal{\mathbf{o}}~g(y)$.
}
\end{example}
\noindent
The concept of $a$-downward O-metrics is not common. One of the usefulness of the concept is to accommodate binary operations that are not commutative (or symmetric) such as division. In addition, downward O-metrics can be seen in real-life applications when considering weights of pairs other than distance types. This is demonstrated in the example below:
\begin{example}
\label{example1}
\textup{
Consider the set $\mathcal{P}$ of friends $\{p_i\}_{i \in \mathbb{Z}}$ and $L(p_i,p_j)$ the degree (or grade) of love between $p_i$ and $p_j$, ranging from $0$ (absolute hate) to $1$ (absolute love). It is reasonable to set $L(p_i,p_j)=1$ if and only if $i=j$, assuming that love is at its highest only when it is for oneself. If we set $L(p_i,p_j)=e^{-|i-j|}$ for any $i,j \in \mathbb{Z}$, $L$ is a $1$-downward $\om$-metric, where $\om(u,v)=\frac{u}{v}$ for $u,v \in I_1:=(0,1]$. 
}
\end{example}
\noindent
An O-metric space $(X,d_\om,a)$ may be neither $a$-upward nor $a$-downward:
\begin{example}
\label{example3} 
\textup{
Let $X=\mathbb{R}$. Let $I_1$ be the interval $(0,\infty)$, $\mathnormal{\mathbf{o}}:I_1 \times I_1 \to I_1$ the function defined
by: 
\begin{equation*}
\mathnormal{\mathbf{o}}(u,v)=\left\{ 
\begin{array}{lll}
\max\left\{\dfrac{u}{v},-\ln(uv)\right\} & \mbox{ if } ~ 0<u,v \leq 1; &  \\ 
\max\left\{ue^v,-\ln u+v\right\} & \mbox{ if } ~ 0<u\leq 1<v; &  \\ 
\max\left\{\dfrac{e^{-u}}{v}, u-\ln v\right\} & \mbox{ if } ~ 0<v \leq 1<u;
&  \\ 
\max\left\{e^{-u+v}, u+v \right\} & \mbox{ if } ~ u,v >1; & 
\end{array}%
\right.
\end{equation*}
and $d_\om:X \times X \to (0,\infty)$ the map defined by: 
\begin{equation*}
d_\om(x,y)=\left\{ 
\begin{array}{lll}
e^{-|x-y|} & \mbox{ if } |x-y| \leq 1; &  \\ 
|x-y| & \mbox{ if } |x-y| > 1. & 
\end{array}
\right.
\end{equation*}
Then $(X,d_\om,1)$ is an $\om$-metric space that is neither a $1$-upward O-metric space nor a $1$-downward O-metric space. 
}
\end{example}
\noindent
In the next section, we discuss the topology induced by an O-metric and some of its properties.

\section{Topology on O-metric spaces}

An O-metric space $(X,d_{\mathnormal{\mathbf{o}}},a)$ is a topological space with the \emph{topology induced by the O-metric} (or \emph{O-metric topology}) defined as:
\begin{equation}
\mathcal{T}_{d_{\mathnormal{\mathbf{o}}}}=\left\{A \subset X: ~ \forall x \in A ~\exists r>0,~  B_{d_\om}(x,r) \subset A \right\},
 \label{topology}
\end{equation}
where 
\begin{equation}
B_{d_\om}(x,r)=\{y \in X: ~ |d_\om(x,y)-a|<r\}
\end{equation}
is the \emph{open ball} centered on $x \in X$ and with radius $r>0$. 
\\
\\
The terminology ``open ball" is employed for consistency with literature on metric-types only: an ``open ball" may not be an open set for $\mathcal{T}_{d_\om}$.  To see this, we refer to the example constructed in 
\cite{openball} of an open ball which is not an open set in a b-metric space which is a special kind of O-metric space. 

\begin{example}[see \cite{openball}]
\textup{
Let $X=\{0\} \cup \left\{\frac{1}{n}: ~n \in \mathbb{N} \right\}$ and
\begin{equation*}
d(x,y)=\left\{
\begin{array}{llll}
0 & \mbox{ if } x=y
\\
1 & \mbox{ if } x \neq y \in \{0, 1\}
\\
|x-y| & \mbox{ if } x \neq y \in \{0\} \cup \left\{ \frac{1}{2n}: ~ n \in \mathbb{N}\right\}
\\
4 & \mbox{ otherwise.}
\end{array}
\right.
\end{equation*}
Then $d$ is a b-metric on $X$ with $s=4$ hence an $\om$-metric space with $\om(u,v)=4(u+v)$ for all $u, v \in [0, \infty)$. The open ball $B_d(1,2)$ is not open in $\mathcal{T}_{d}$. In fact, the O-metric topology cannot be induced by a metric (for example, the singleton $\{1\}$ is an open set).
}
\end{example}
\noindent
The following proposition holds:
\begin{proposition}\label{tough}
Let $\{x_n\}_{n \in \mathbb{N}}$ be a sequence of points in an O-metric space $(X,d_{\om},a)$, such that $\displaystyle \lim_{n \to \infty} d_\om(x_n,x)=a$ for some $x \in X$; then, $\{x_n\}$ converges (in the O-metric topology) to $x \in X$. 
\end{proposition}
\begin{proof}
Let $\{x_n\}$ be a sequence of points in $X$, and $x \in X$ be such that $\lim_{n \to \infty} d_\om(x_n,x)=a$. Let $A$ be an open set (in the O-metric topology) containing $x$. Then, there is $r>0$ such that $B_{d_\om}(x,r) \subset A$. For such $r>0$, there is $n_0 \in \mathbb{N}$ such that $|d_\om(x_n,x)-a|< r$ for all $n \geq n_0$. Thus, $x_n \in B_{d_\om}(x,r) \subset A$ for $n \geq n_0$. This means that $\{x_n\}$ converges to $x$ in the topology $\mathcal{T}_{d_{\mathnormal{\mathbf{o}}}}$.
\end{proof}
\noindent
It is not known whether the converse always hold. Suppose $\{x_n\}$ is a convergent sequence in the O-metric topology, and $x$  a limit of $\{x_n\}$; then, 
\begin{equation}\label{GreatGod}
\mbox{for any open set $A$, there is $n_0 \in \mathbb{N}$ such that $x_n \in A$ for all $n \geq n_0$.}
\end{equation}
\noindent
For any $\epsilon>0$, $B_{d_\om}(x,\epsilon)$ is not necessarily an open set, hence, the substitution $A=B_{d_\om}(x,\epsilon)$ in (\ref{GreatGod}) (which would mean that $\lim_{n \to \infty}d_\om(x_n,x)=a$) is not allowed. However, if the following property holds,
\begin{equation}
\mbox{for all $\epsilon>0$, there is an open set $A_\epsilon$ containing $x$ such that $A_\epsilon \subset B_{d_\om}(x,\epsilon)$,}
\end{equation}
then $\lim_{n \to \infty}d_\om(x_n,x)=a$.
\\
\\
\noindent
In view of Proposition \ref{tough}, we state the following definition:
\begin{definition}
Let $(X,d_\om,a)$ be an O-metric space. A sequence $\{x_n\}_{n \in 
\mathbb{N}}$ of points in $X$ is said to O-converge to $x \in X$ if $\lim_{n \to \infty} d_\mathnormal{%
\mathbf{o}}(x_n,x)=a$. In such case, $x$ is called the O-limit of $\{x_n\}$, and we write $x_n \xrightarrow{\text{O}} x$. 
\end{definition}
\noindent
In the remaining part of the article, when no confusion arises, ``convergence" will mean ``O-convergence", ``limit" will mean ``O-limit" , and ``$x_n \to x$" will mean ``$x_n \xrightarrow{\text{O}} x$".  
\\
\\
It should be noted that the O-limit of an O-convergent sequence may not be unique as seen in the example below:
\begin{example}
\label{olala} Consider the O-metric space $(X,d_\mathnormal{\mathbf{o}},a)$
where $X=[-1,1]$, $~~a=1$, 
\begin{equation}
\label{GodisGreat}
d_\mathnormal{\mathbf{o}}(x,y)=\left\{%
\begin{array}{lll}
1 & \mbox{ if } x=y &  \\ 
|xy| & \mbox{ if } x \neq y & 
\end{array}%
\right.
\end{equation}
and 
\begin{equation}
\label{GodisGreat1}
\mathnormal{\mathbf{o}}(u,v)=\left\{%
\begin{array}{lll}
\frac{1}{uv} & \mbox{ if } u,v \neq 0, &  \\ 
1 & \mbox{ if } u=0 \mbox{ or } v=0. & 
\end{array}%
\right.
\end{equation}
\noindent
Consider the sequence $\{x_n\}_{n \in \mathbb{N}}$ defined by $%
x_n=1-\frac{1}{n}$, for all $n \in \mathbb{N}$. For $x \in \{-1,1\}$, $d_%
\mathnormal{\mathbf{o}}(x_n,x)=\left(1-\frac{1}{n}\right)|x|=1-\frac{1}{n}
\to 1$ as $n \to \infty$. Hence $\pm 1$ are both O-limits of the sequence $\{x_n\}$.
\end{example}
\noindent
The following definition of Cauchy sequences in O-metric spaces is adopted:
\begin{definition}
Let $(X,d_\om,a)$ be an O-metric space. A sequence $\{x_n\}_{n \in 
\mathbb{N}}$ of points in $X$ is said to be a Cauchy sequence if $\lim_{n,m \to \infty} d_\mathnormal{\mathbf{o}}%
(x_n,x_m)=a$. 
\end{definition}
\noindent
Interestingly, unlike the metric-type spaces in Example \ref{referencemet}, an O-convergent sequence in an O-metric space may not even be a Cauchy sequence:

\begin{example}\label{AIMS}
\textup{
Define for pairs $(x,y)$ of positive real numbers, the function:
\begin{equation*}
d(x,y)=\left\{
\begin{array}{lll}
\frac{2xy}{x^2+y^2}, & \mbox{if } x \neq y
\\
0, & \mbox{if } x =y.
\end{array}
\right.
\end{equation*}
\noindent
Given that $2xy \leq x^2+y^2$ with equality if and only if $x=y$, we have that $d(x,y) \in [0,1)$ for any $x,y>0$. 
In fact, for distinct positive real numbers $x,y, z$, $d(x,y) \leq \frac{2d(x,z)}{d(z,y)}$, since
\begin{equation*}
\begin{array}{lllccclll}
d(x,z) &=& \frac{2xz}{x^2+z^2} 
\\
&=& \frac{2xy}{x^2+y^2}\cdot \frac{2yz}{y^2+z^2}\cdot \frac{(x^2+y^2)(y^2+z^2)}{2y^2(x^2+z^2)}
\\
&=& d(x,y)d(y,z) \frac{x^2 y^2+x^2 z^2+y^4+y^2 z^2}{2(x^2 y^2 +y^2 z^2)}
\\
& \geq & \frac{1}{2}d(x,y)d(y,z).
\end{array}
\end{equation*}
In the case where $x,y,z$ are not all distinct, $d(x,y) \leq d(x,z)+d(z,y)$. Therefore,
$$d(x,y) \leq \om(d(x,z),d(z,y))$$ for all $x,y,z \in X$, where $\om:I_0 \times I_0 \to \mathbb{R}_+$, with $I_0=[0,1)$, is defined by
\begin{equation*}
\om(u,v)=\left\{\begin{array}{lll} \frac{2u}{v}, & \mbox{ if   $uv \neq 0$} \\ u+v, &\mbox{ otherwise}. \end{array} \right.
\end{equation*}
Therefore $(X,d,0)$ is an $\om$-metric space, with $X=(0,\infty)$.
\\
Let $\{x_n\}$ be the sequence in $X$ such that $x_n =\frac{1}{n}$ for all $n \in \mathbb{N}$. 
\\
Then $x_n \xrightarrow{\text{O}} x$ for any $x >0$, since
\begin{equation*}
d(x_n,x)=\frac{\frac{2x}{n}}{\frac{1}{n^2}+x^2}=\frac{2x}{n} \cdot \frac{n^2}{1+n^2 x^2}=\frac{2nx}{1+n^2x^2} \to 0 ~\mbox{ as $n \to \infty$}.
\end{equation*}
However, $\{x_n\}$ is not a Cauchy sequence since for any $n,m \in \mathbb{N}$,
\begin{equation*}
d(x_n,x_m)=\frac{\frac{2}{nm}}{\frac{1}{n^2}+\frac{1}{m^2}}=\frac{2nm}{n^2+m^2} \nrightarrow 0 \mbox{ as } n,m \to 0.
\end{equation*}
}
\end{example}
\noindent
The definition of O-completeness is introduced as follows:
\begin{definition}
An O-metric space $(X,d_\om,a)$ is called O-complete if the Cauchy sequences in $X$ are the O-convergent sequences in $X$.
\end{definition}
\noindent
Thus, in an O-complete O-metric space, every O-convergent sequence is a Cauchy sequence and every Cauchy sequence is O-convergent.
\\
\\
Next, we state some conditions on the binary operation $\om$ for the O-metric topology on an $\om$-metric space $(X,d_\om,a)$ to be metrizable, Hausdorff, and such that O-convergent sequences have unique O-limits. For properties of the O-metric topology, the focus will be on upward O-metrics  since the O-metric topology on any O-metric space is induced by some upward O-metric:

\begin{theorem}
\label{surprise} 
The topology of an O-metric space $X$ is induced by an upward O-metric on $X$.
\end{theorem}

\begin{proof}
Let $(X,d_\om,a)$ be an O-metric space, with $a \in \mathbb{R}_+$, and $I_a$ an interval of non-negative numbers containing $a$. Let $J_a:=I_a \cap [a,\infty)$, and define the maps $\xi:J_a \times J_a \to \mathbb{R}_+$ and $d_\xi: X \times X \to J_a$ by:
\begin{equation}\label{equationolala}
\xi(u,v)=\max\left\{\om(u,v),\om(u,2a-v),\om(2a-u,v),\om(2a-u,2a-v),2a\right\} ~~ \forall u,v \in J_a,
\end{equation}
\begin{equation}\label{equationola}
d_\xi(x,y)=a+|d_\om(x,y)-a| ~~ \forall x,y \in X.
\end{equation}
$(X,d_\xi,a)$ is an $a$-upward O-metric space and the topology induced by $d_\om$ is the same as the topology induced by $d_\xi$ as $B_{d_\om}(x,r)=B_{d_\xi}(x,r)$ for all $x \in X$ and $r>0$.
\end{proof}

\noindent
\begin{remark}
\textup{
It should be noted that the function $\xi$ in (\ref{equationolala}) was only defined to show that $d_\xi$ defined in (\ref{equationola}) is indeed an upward O-metric. Consider for example, the downward O-metric space defined in Example \ref{olala}. Given $d_\om$  defined in (\ref{GodisGreat}) for all $x,y \in [-1,1]$ as
$$
d_\mathnormal{\mathbf{o}}(x,y)=\left\{%
\begin{array}{lll}
1 & \mbox{ if } x=y &  \\ 
|xy| & \mbox{ if } x \neq y, & 
\end{array}%
\right.$$
using equation (\ref{equationola}), $d_\xi(x,y)=2-|xy|$ for all $x,y \in [-1,1]$. Given the function $\om$ defined in (\ref{GodisGreat1}) by $$\mathnormal{\mathbf{o}}(u,v)=\left\{%
\begin{array}{lll}
\frac{1}{uv} & \mbox{ if } u,v \neq 0, &  \\ 
1 & \mbox{ if } u=0 \mbox{ or } v=0,& 
\end{array}%
\right.$$ it is hard to construct $\xi$ using equation (\ref{equationolala}). However, it is easy to show directly that $d_\xi$ is an $\chi$-metric, with $\chi(u,v)=2(u+v-1)-uv$ for all $u,v \in [1,2]$. Thus $d_\xi$ is an upward $\xi$-metric, but also an upward $\chi$-metric. 
}


%
%
%
%
\end{remark}

\subsection{Conditions for Metrizability}

\noindent
\begin{theorem}\label{metricequiv}\textup{\cite{ourpaper}}
Let $X$ be a non-empty set, $a \in \mathbb{R}_+$ a non-negative real number, $\om:[a,\infty) \times [a,\infty) \to  [a,\infty)$ a binary operation on $[a,\infty)$, and $\lambda:[a,\infty) \to [0,\infty)$ a function such that the following properties are satisfied:
\begin{itemize}
\item[$(E_1)$]\label{E1} $\lambda$ is increasing, with $\lambda(a)=0$;

\item[$(E_2)$]\label{E2} $\lambda(\om(u,v))=\lambda(u)+\lambda(v) ~ \forall u,v \geq a.$
\end{itemize}
$(X,d_\om,a)$ is an upward O-metric space if and only if $(X,\lambda \circ d_\om)$ is a metric space. 
In such case, the O-metric topology on $(X,d_\om,a)$ and the metric topology on $(X,\lambda \circ d_\om)$ coincide.
\end{theorem}

\begin{proof}  Let $(X,d_\om,a)$ be an upward O-metric space. Under condition $(E_1)$, $\lambda(u)=0$ if and only if $u=a$, for all $u \geq a$. Let $x,y,z \in X$; we have that:
\begin{equation*}
\begin{array}{lcl}
\lambda (d_\om (x,y)) = 0 \Leftrightarrow d_\om(x,y)=a \Leftrightarrow x=y,
\\
\lambda(d_\om(x,y))=\lambda(d_\om(y,x)),
\\
\lambda(d_\om(x,z)) \leq \lambda\left(\om(d_\om(x,y),d_\om(y,z))\right) =\lambda(d_\om(x,y))+\lambda(d_\om(y,z)).
\end{array}
\end{equation*}
Therefore, $\lambda \circ d_\om$ is a metric on $X$.
\\
The converse holds: if $\lambda \circ d_\om$ is a metric on $X$, then, $d_\om$ is an $a$-upward O-metric on $X$. Indeed, under conditions $(E_1)$ and $(E_2)$, the function $\lambda:[a,\infty) \to {\rm Im}(\lambda)$ is bijective, $\lambda^{-1}(\lambda(u))=u$, and $\lambda^{-1}(\lambda(u)+\lambda(v))=\om(u,v)$ for all $u \geq a$; therefore, if $x,y,z \in X$,
\begin{equation*}
\begin{array}{lcl}
d_\om (x,y) = a \Leftrightarrow \lambda(d_\om(x,y))=0 \Leftrightarrow x=y,
\\
d_\om(x,y)=\lambda^{-1}\left(\lambda(d_\om(x,y))\right)=\lambda^{-1}\left(\lambda(d_\om(y,x))\right)=d_\om(y,x),
\end{array}
\end{equation*}
\begin{equation*}
\begin{array}{lcl}
d_\om(x,z) &=& \lambda^{-1}\left(\lambda(d_\om(x,y))\right) 
\\
&\leq& \lambda^{-1}\left(\lambda(d_\om(x,y))+\lambda(d_\om(y,z))\right)
\\
&=& \om(d_\om(x,y),d_\om(y,z))
\end{array}
\end{equation*}
\noindent
Now, let $\mathcal{T}_{d_\om}$ be the O-metric topology on the space $(X,d_\om,a)$, where $\om$ satisfies conditions $(E_1)$ and $(E_2)$. Then, for $x \in X$ and $r>0$, the open ball $B_{d_\om}(x,r)$ is given by:
\begin{equation*}
\begin{array}{lcl}
B_{d_\om}(x,r)&=&\{ y \in X: ~ |d_\om(x,y)-a|<r \}
\\
&=& \{y \in X:~ a \leq d_\om(x,y)<a+r\} ~~ \mbox{(since $d_\om$ is an $a$-upward O-metric space)}
\\
&=& \{ y \in X: ~ 0 \leq \lambda(d_\om(x,y))<\lambda(a+r)\}
\\
&=& B_{\lambda\circ d_\om}(x,\lambda(a+r)),
\end{array}
\end{equation*}
where $B_{\lambda\circ d_\om}(x,\lambda(a+r))$ is the open ball centered on $x \in X$ with radius $r$, relative to the metric topology on $(X,\lambda \circ d_\om)$. As $\lambda$ is bijective, we can conclude that the open sets in the O-metric topology on $(X,d_\om,a)$ are exactly the open sets in the metric topology on $(X,\lambda \circ d_\om)$.
\end{proof}

\begin{example}
Consider the function $\lambda:[1,\infty) \to [0,\infty)$ defined by $\lambda(t)=\ln(t)$ for all $t \geq 1$. Since $\lambda$ is increasing such that $\lambda(1)=0$ and $\lambda(uv)=\lambda(u)+\lambda(v)$ for all $u,v \in [1,\infty)$, conditions $(E_1)$ and $(E_2)$ of Theorem \ref{metricequiv} are satisfied, with $\om$ defined by $\om(u,v)=uv$ for all $u,v \geq 1$. Therefore, the topology on a multiplicative metric space  $(X,d_\times,1)$ coincides with the metric topology on the assciated metric space $(X,\lambda \circ d_\times)$.
\end{example}
\noindent
From Theorem \ref{metricequiv}, we can state the following proposition:
\begin{proposition}
Let $(X,d_\om,a)$ be an $a$-upward $\om$-metric space such that $\om$ satisfies conditions $(E_1)$ and $(E_2)$ for some function $\lambda:[a,\infty) \to [0,\infty)$. A sequence $\{x_n\}$ of points in $X$ converges in the O-metric topology to some point $x \in X$ if and only if  $x_n \xrightarrow{\text{O}} x$. 
\end{proposition}

\begin{proof}
We known from Proposition \ref{tough} that $x_n \xrightarrow{\text{O}} x$ implies that $\{x_n\}$ converges to $x$ in the O-metric topology. Conversely, suppose that $\{x_n\}$ converges to $x$ in the O-metric topology. From Theorem \ref{metricequiv}, $\{x_n\}$ converges to $x$ in the metric topology on $(X,\lambda \circ d_\om)$, hence $\lim_{n \to \infty} \lambda(d_\om(x_n,x))=0$. Given $\epsilon>0$, there is $n_0 \in \mathbb{N}$ such that $\lambda(d_\om(x_n,x))<\lambda(\epsilon)$, so $d_\om(x_n,x) <\epsilon$, for all $n \geq n_0$. Thus $\lim_{n \to \infty} d_\om(x_n,x))=a$, and  $x_n \xrightarrow{\text{O}} x$. 
\end{proof}

\subsection{Conditions for openness of open balls}

\noindent
The following theorem states conditions under which open balls are open sets in the O-metric topology.

\begin{theorem}
\label{conly} Let $(X,d_\mathnormal{\mathbf{o}},a)$ be an $a$-upward $\om$-metric
space, where $\om$ verifies the following conditions:
\begin{itemize}
\item[$(C_1)$] there exists $\gamma:[a,\infty) \times [a,\infty) \to \mathbb{%
R}$ such that $a \leq u <r \implies \left\{ 
\begin{array}{lll}
\gamma(r,u)>a &  &  \\ 
\mathnormal{\mathbf{o}}(u,\gamma(r,u))\leq r; &  & 
\end{array}%
\right.$

\item[$(C_2)$] $\mathnormal{\mathbf{o}}$ is increasing in both variables.
\end{itemize}
Then, every open ball is an open set in the O-metric topology, and the O-metric topology on $X$ is Hausdorff.
\end{theorem}

\begin{proof}
\item
Let $B_{d_\om}(x_0,r)$ be the open ball centered on $x_0 \in X$ with radius $r>0$. To show that $B_{d_\om}(x_0,r)$, we show that for any $x \in B_{d_\om}(x_0,r)$, there is $s>0$ such that $B_{d_\om}(x,s) \subset B_{d_\om}(x_0,r)$.
\\
Let $x \in B_{d_\om}(x_0,r)$. Then $a \leq d_\om(x,x_0)<a+r$, and by $(C_1)$, 
\begin{equation}\label{referencia}
\left\{
\begin{array}{lll}
\gamma(a+r,d_\om(x,x_0))>a \\
\om(d_\om(x,x_0),\gamma(a+r,d_\om(x,x_0))) \leq a+r.
\end{array}
\right.
\end{equation}
If we let $s:=\gamma(a+r,d_\om(x,x_0))-a$, then $s>0$, and if $y \in B_{d_\om}(x,s)$ (i.e., if $a\leq d_\om(x,y)<a+s$), then $y \in B_{d_\om}(x_0,r)$, since, by combining the triangle $\om$-inequality, property $(C_2)$, and (\ref{referencia}),
$$
\begin{array}{lll}
a \leq d_\om(x_0,y) &\leq& \om(d_\om(x_0,x),d_\om(x,y)) 
\\
&<&\om(d_\om(x_0,x),a+s)=\om(d_\om(x_0,x),\gamma(a+r,d_\om(x,x_0))) 
\\
&\leq& a+r.
\end{array}
$$
The open ball $B_{d_\om}(x_0,r)$ is therefore an open set, for any $x_0 \in X$ and $r>0$.
\item
Let $x,y \in X$ be such that $x \neq y$. Take $U:=B_{d_\om}(x,q)$ and $V:=B_{d_\om}(y,r)$, where $q=\gamma(d_\om(x,y),a+r)-a$. Then $U$ and $V$ are disjoint open sets such that $x \in U$ and $y \in V$. To show that $U\cap V=\emptyset$, we prove by contradiction: if $z \in B_{d_\om}(x,q) \cap B_{d_\om}(y,r)$, then $d_\om(x,z)<a+ q=\gamma(d_\om(x,y),a+r)$ and $d_\om(y,z)<a+r$, so $d_\om(y,x) \leq \om(d_\om(y,z),d_\om(z,x)) < \om(a+r,\gamma(d_\om(x,y),a+r)) \leq d_\om(x,y)$ from $(C_1)$, which is absurd. Therefore, the O-metric topology is Hausdorff.
\end{proof}
\noindent
The conditions $(C_1)$ and $(C_2)$ also guarantee that O-convergence and convergence in the O-metric topology coincide, as shown in the following proposition:
\begin{proposition}
Let $(X,d_\om,a)$ be an $a$-upward $\om$-metric space such that $\om$ satisfies conditions $(C_1)$ and $(C_2)$. A sequence $\{x_n\}$ of points in $X$ converges in the O-metric topology to some point $x \in X$ if and only if  $x_n \xrightarrow{\text{O}} x$. 
\end{proposition}

\begin{proof}
We know from Proposition \ref{tough} that $x_n \xrightarrow{\text{O}} x$ implies that $\{x_n\}$ converges to $x$ in the O-metric topology. Conversely, suppose that $\{x_n\}$ converges to $x$ in the O-metric topology. Then, for any $\epsilon>0$, given that $B_{d_\om}(x,r)$ is an open set, there is $n_0 \in \mathbb{N}$ such that $x_n \in B_{d_\om}(x,\epsilon)$, i.e., $d_\om(x_n,x)<\epsilon$, for all $n \geq n_0$. Therefore, $\lim_{n \to \infty} d_\om(x_n,x)=a$, and  $x_n \xrightarrow{\text{O}} x$. 
\end{proof}

\subsection{Conditions for uniqueness of O-limits}

The O-limit of an O-convergent sequence needs not be unique, as seen in Example \ref{olala}. We specify sufficient conditions under which the O-limit of a sequence, if it exists, is unique. 

\begin{proposition}
\label{UUU} Let $(X,d_\mathnormal{\mathbf{o}},a)$ be an $\om$-metric space (not necessarily $a$-upward) such that $\om$ satisfies the following conditions:
\begin{itemize}
\item[$(U_1)$]\label{U1} $\mathnormal{\mathbf{o}}$ is continuous at $(a,a)$; 
\item[$(U_2)$]\label{U2} $\mathnormal{\mathbf{o}}$ is nondecreasing in both variables
and either $\mathnormal{\mathbf{o}}(u,a)=a \Leftrightarrow u=a$ for all $u \in I_a$, or $%
\mathnormal{\mathbf{o}}(a,u)=a \Leftrightarrow u=a$ for all $u \in I_a$.
\end{itemize}
Then, the O-limit of an O-convergent sequence is unique  
\end{proposition}

\begin{proof} Assume $(U_1)$ and $(U_2)$ holds. From condition $(U_2)$, $\om(a,a)=a$. Now, suppose $\lim_{n \to \infty} d_\om(x_n,x)=\lim_{n \to \infty} d_\om(x_n,y)= a$, where $x,y \in X$, and $\{x_n\}$ is a sequence of points in $X$.  By the triangle $\om$-inequality, for all $n\in \mathbb{N}$,
\begin{equation}
\left\{
\begin{array}{lll}
d_\om(x,y) &\leq& \om(d_\om(x,x_n),d_\om(x_n,y)). 
\\
d_\om(x_n,x) &\leq& \om(d_\om(x,y),d_\om(x_n,y))
\\
d_\om(x_n,x) &\leq& \om(d_\om(x_n,y),d_\om(y,x))
\end{array}
\right.
\end{equation}
Taking the limit as $n \to \infty$, since $\om$ is continuous at $(a,a)$, then $d_\om(x,y)  \leq a$, and
$$\left\{
\begin{array}{lcl}
a \leq \om(d_\om(x,y),a) \leq \om(a,a)=a
\\
a \leq \om(a,d_\om(x,y)) \leq \om(a,a)=a
\end{array}
\right.$$
Therefore, $\om(d_\om(x,y),a)=\om(a,d_\om(x,y))=a$, hence, by $(U_2)$, $d_\om(x,y))=a$ and $x=y$. The O-limit of any O-convergent sequence $\{x_n\}$ is therefore unique.
\end{proof}
\noindent 
One can easily verify that metric spaces,
b-metric spaces, multiplicative metric spaces, b-multiplicative metric
spaces, ultrametric spaces, $\theta$-metric spaces and p-metric spaces
satisfy conditions $(U_1)$ and $(U_2)$, hence the limit of a convergent
sequence in any of these spaces is unique. Also, the O-metric space in Example \ref{genl} satisfy conditions $%
(U_1)$ and $(U_2)$. 
\\
\\
Note that the function $\mathnormal{\mathbf{o}}$ is not unique for
any O-metric space $(X,d_\mathnormal{\mathbf{o}},a)$ given the fact  that any other
function which dominates $\mathnormal{\mathbf{o}}$ will serve the purpose. However, by virtue of Proposition \ref{UUU}, if $\tilde{\om}:I_a \times I_a \to \mathbb{R}_+$ is a function such that $\tilde{\om}(u,v) \geq \om(u,v)$ for all $u,v \in I_a$ , and $\om$ does not satisfy $(U_1)$ or $(U_2)$, then $\tilde{\om}$ cannot satisfy conditions $(U_1)$ and $(U_2)$ at the same time.
\\
\\
The following result shows that under conditions $(U_1)$ and $(U_2)$, every O-convergent sequence in an upward O-metric space is a Cauchy sequence.
\begin{proposition}
Let $(X,d_\mathnormal{\mathbf{o}},a)$ be an upward $\om$-metric space such that $\om(a,a)=a$ and $\om$ is continuous at $(a,a)$. Then, every O-convergent sequence in $X$ is a Cauchy sequence.
\end{proposition}

\begin{proof}
The result follows from the triangle $\om$-inequality. Indeed, if $\{x_n\}$ is an O-convergent sequence and $x$ is its O-limit, then $a \leq d_\om(x_n,x_m) \leq \om(d_\om(x_n,x),d_\om(x,x_m))$ for all $n,m \in \mathbb{N}$. As $n,m \to \infty$,  $d_\om(x_n,x_m) \to a$, hence $\{x_n\}$ is a Cauchy sequence.
\end{proof}
\noindent
The $0$-upward O-metric space in Example \ref{AIMS} satisfies neither condition $(U_1)$ nor $(U_2)$: the function $\om$ is not continuous at $(0,0)$ and $\om$ is not nondecreasing in the second variable. 
\\
\\
We conclude this subsection by stating the following proposition which is a direct consequence of the definition of O-convergence.

\begin{proposition}
\label{sequential} Let $(X_1,d_{\mathnormal{\mathbf{o}}_1},a)$ and $(X_2,d_{%
\mathnormal{\mathbf{o}}_2},b)$ be two O-metric spaces, and $f:X_1 \to X_2$ a mapping from $X_1$ to $X_2$. The following are equivalent:

\begin{itemize}
\item[(i)] $f$ is sequentially continuous \footnote{Sequential continuity here is relative to O-convergence.} at $\tilde{x} \in X_1$, that is, for
any sequence $\{x_n\}_{n \in \mathbb{N}}$ of points in $X_1$, $x_n  \xrightarrow{\text{O}} 
\tilde{x} \implies f(x_n)  \xrightarrow{\text{O}}  f(\tilde{x})$;

\item[(ii)] $\forall \epsilon>0$, $\exists \delta>0:$ $|d_{\mathnormal{%
\mathbf{o}}_1}(x,\tilde{x})-a| < \delta \implies |d_{\mathnormal{\mathbf{o}}%
_2}(f(x),f(\tilde{x}))-b|<\epsilon$.
\end{itemize}
\end{proposition}

\begin{example}
Consider two O-metric spaces $(X_1,d_{\mathnormal{\mathbf{o}}_1},a)$ and $%
(X_2,d_{\mathnormal{\mathbf{o}}_2},b)$ such that $X_1=X_2=\mathbb{R}$, $a=1$%
, $b=0$, $\mathnormal{\mathbf{o}}_1(u,v)=\frac{u}{v}$ for all $u \geq 0$
and $v>0$, $\mathnormal{\mathbf{o}}_2(u,v)=(u+1)(v+1)$ for all $u,v \geq 0$, 
$d_{\mathnormal{\mathbf{o}}_1}(x,y)=e^{-|x-y|}$ and $d_{\mathnormal{\mathbf{o%
}}_2}(x,y)=\ln(1+|x-y|)$ for all $x,y \in \mathbb{R}$. The map $f:X_1 \to X_2
$ defined by $f(x)=x^2-2$ is sequentially continuous at any $\tilde{x} \in 
\mathbb{R}$ and  it satisfies tthe condition (ii) of the Proposition \ref{sequential}, with $%
\delta=1-e^{|\tilde{x}|-\sqrt{|\tilde{x}|^2+e^\epsilon-1}}$ for any $%
\epsilon>0$.
\end{example}

%
\noindent


\noindent
In the next section, we establish polygon inequalities for points in an O-metric space from the triangle $\om$-inequality, and define a natural generalization of series. 

\section{Polygon $\om$-inequalities and $\om$-series}

One of the beauties of the notion of O-metric spaces lies in the modification of the triangle inequality axiom of a metric space to accommodate other binary operations which are not necessarily associative. The triangle $\mathnormal{\mathbf{o}}$-inequality obtained becomes interesting to study when applied to more than three points. 
\\
\\
Let $(X,d_{\mathnormal{\mathbf{o}}},a)$ be an O-metric space.  
For $n \in \mathbb{N}$, let $x_0,  x_1, x_2,\ldots,x_{n+1}$ be a finite sequence of $n+2$ points in $X$. The triangle $\mathnormal{\mathbf{o}}$-inequality provides many upper bounds for $d_{\mathnormal{\mathbf{o}}}(x_0,x_{n+1})$ as functions of exactly $d_{\mathnormal{\mathbf{o}}}(x_0,x_1), d_{\mathnormal{\mathbf{o}}}(x_1,x_2), \ldots, d_{\mathnormal{\mathbf{o}}}(x_{n},x_{n+1})$ with each of $d_{\mathnormal{\mathbf{o}}}(x_0,x_1), d_{\mathnormal{\mathbf{o}}}(x_1,x_2), \ldots, d_{\mathnormal{\mathbf{o}}}(x_{n},x_{n+1})$  occurring exactly once, without interchanging the order. In fact, if we view $\om$ as a binary operation, the upper bounds exactly correspond to the expressions $d_{\mathnormal{\mathbf{o}}}(x_0,x_1)~ \mathnormal{\mathbf{o}}~  d_{\mathnormal{\mathbf{o}}}(x_1,x_2)  ~\mathnormal{\mathbf{o}}~  \cdots ~\mathnormal{\mathbf{o}}~  d_{\mathnormal{\mathbf{o}}}(x_{n},x_{n+1})$  which depends on how  parentheses are placed. The expression  $d_{\mathnormal{\mathbf{o}}}(x_0,x_1) ~\mathnormal{\mathbf{o}}~ d_{\mathnormal{\mathbf{o}}}(x_1,x_2) ~ \mathnormal{\mathbf{o}}~  \cdots ~ \mathnormal{\mathbf{o}} ~ d_{\mathnormal{\mathbf{o}}}(x_{n},x_{n+1})$ has at most $C_{n}$ values, where  $C_{n}:=\frac{1}{n+1} {{2n}\choose{n}}$ is the $n$-th Catalan number (see \cite{assoc1,assoc2}). %
The lemma below immediately follows:
\begin{lemma}\label{polygon} 
Denote by $\Omega_{n,a}$ (or simply $\Omega_n$ when no confusion arises) the set (of order at most equal to $C_n$) of functions $\Delta_n: I_a^{n+1} \to I_a$ such that $\Delta_n(t_0,t_1,\ldots,t_{n})=t_0 ~ \mathnormal{\mathbf{o}} ~ t_1 ~\mathnormal{\mathbf{o}} \cdots \mathnormal{\mathbf{o}} ~t_n$. The triangle $\mathnormal{\mathbf{o}}$-inequalities involving $n+2$ points $x_0,x_1,\ldots,x_{n+1}$ in 
an O-metric space $(X,d_{\mathnormal{\mathbf{o}}},a)$ become:
\begin{equation} \label{comp11}
d_{\mathnormal{\mathbf{o}}}(x_0,x_{n+1}) \leq \Delta_n\left(d_{\mathnormal{\mathbf{o}}}(x_0,x_1),d_{\mathnormal{\mathbf{o}}}(x_1,x_2),\ldots,d_{\mathnormal{\mathbf{o}}}(x_{n},x_{n+1})\right)   ~ \forall \Delta_n \in \Omega_n,
\end{equation}
or simply
\begin{equation} \label{comp12}
d_{\mathnormal{\mathbf{o}}}(x_0,x_{n+1}) \leq d_{\mathnormal{\mathbf{o}}}(x_0,x_1) ~\mathnormal{\mathbf{o}}~ d_{\mathnormal{\mathbf{o}}}(x_1,x_2) ~\mathnormal{\mathbf{o}} \cdots \mathnormal{\mathbf{o}} ~ d_{\mathnormal{\mathbf{o}}}(x_{n},x_{n+1}),
\end{equation}
and are called polygon $\mathnormal{\mathbf{o}}$-inequalities.
\end{lemma}
\noindent
\begin{example}
\label{detailex}
\textup{
Let $\om$ be a function defined for pairs of elements in the interval $I_0=[0,\infty)$ by $\om(u,v)=u+2v$. Then, for $0 \leq n \leq 3$, the set $\Omega_n$ is given by:
\begin{equation*}
\begin{array}{llllccccllll}
\Omega_0&=& \{\Delta_0\} & \mbox{ where }  \Delta_0(t_0)=t_0 
\\
\\
\Omega_1&=&\{\Delta_1\} & \mbox{ where }  \Delta_1(t_0,t_1)=\om(t_0,t_1)=t_0+2t_1 
\\
\\
\Omega_2&=& \{\Delta_2^1,\Delta_2^2\} & \mbox{ where } \Delta_2^1(t_0,t_1,t_2)=t_0 \om (t_1 \om t_2) = t_0+2t_1+4t_3 
\\
&&&   ~~~~~~~~~\Delta_2^2(t_0,t_1,t_2)=(t_0\om t_1) \om t_2 = t_0+2t_1+2t_2 
\\
\\
\Omega_3&=& \{\Delta_3^j:~ 1\leq j \leq 5\} & \mbox{ where } \Delta_3^1(t_0,t_1,t_2,t_3)=t_0 \om (t_1 \om (t_2 \om t_3))=t_0+2t_1+4t_2+8t_3
\\
&&&   ~~~~~~~~~  \Delta_3^2(t_0,t_1,t_2,t_3)=t_0 \om ((t_1 \om t_2) \om t_3)=t_0+2t_1+4t_2+4t_3
\\
&&&   ~~~~~~~~~  \Delta_3^3(t_0,t_1,t_2,t_3)=(t_0 \om t_1) \om (t_2 \om t_3)=t_0+2t_1+2t_2+4t_3
\\
&&&   ~~~~~~~~~  \Delta_3^4(t_0,t_1,t_2,t_3)=(t_0 \om (t_1 \om t_2)) \om t_3=t_0+2t_1+4t_2+2t_3
\\
&&&   ~~~~~~~~~  \Delta_3^5(t_0,t_1,t_2,t_3)=((t_0 \om t_1) \om t_2) \om t_3=t_0+2t_1+2t_2+2t_3
\end{array}
\end{equation*}
}
\end{example}
\noindent
It should be noted that if $\mathnormal{\mathbf{o}}$ is associative as a binary operation, then there is only one inequality (\ref{comp11} - \ref{comp12}). This is the case when $\mathnormal{\mathbf{o}}$ is the addition as in a metric space $(X,d)$, with $d(x_0,x_{n+1}) \leq \sum_{i=0}^n d(x_i,x_{i+1})$, or when $\om$ is the multiplication as in multiplicative metric spaces $(X,d_\times)$, with  $d_\times(x_0,x_{n+1}) \leq \prod_{i=0}^n d_\times(x_i,x_{i+1})$. 
\\
\\
In the non-associative case, it becomes necessary to define patterns that allow a function $\Delta_n$ to be expressed in function of some $\Delta_p$ and $\Delta_q$, where $p+q=n$. 

\subsection{Patterns and generalized series}

\begin{definition}\label{pattern1}
Let $\{\alpha_n\}_{n \in \mathbb{N}}$ be a sequence of integers such that $1 \leq \alpha_n \leq n-1$ for all $n \in \mathbb{N}$. A sequence $\{h_n\}_{n \in \mathbb{N}}$ of functions $h_n \in \Omega_{n-1}$ is said to follow the pattern of integers $\{\alpha_n\}_{n \in \mathbb{N}}$ if \begin{equation}
h_n(t_1,t_2,\ldots,t_n) = h_{\alpha_n}(t_1,t_2,\ldots,t_{\alpha_n}) ~\mathnormal{\mathbf{o}}~ h_{n-\alpha_n}(t_{\alpha_n+1},\ldots,t_n) \mbox{ for all $n \geq 2$.}
\end{equation}
\end{definition}
 \noindent
We consider the following example:
\begin{example}
Let $\{u_n\}$, $\{v_n\}$, $\{w_n\}$ and $\{z_n\}$ be sequences of functions 
defined as follow: If $n \in \mathbb{N}$ and $(t_1,t_2,\ldots,t_n) \in I_a^n$, then $u_1(t_1)=v_1(t_1)=w_1(t_1)=z_1(t_1)=t_1$, and for $n \geq 2$,
\begin{equation}\label{recursion}
\left\{
\begin{array}{lllll}
u_n(t_1,t_2,\ldots,t_n) &=& u_{n-1}(t_1,t_2,\ldots,t_{n-1}) ~\mathnormal{\mathbf{o}}~ t_n, 
\\
v_n(t_1,t_2,\ldots,t_n) &=& v_{p}\left(t_1,t_2,\ldots,t_{p}\right)~\mathnormal{\mathbf{o}} ~v_{q}\left(t_{p+1},t_{p+2},\ldots,t_{n}\right), 
\\
w_n(t_1,t_2,\ldots,t_n) &=& w_{2^{l-1}}\left(t_1,\ldots,t_{2^{l-1}}\right) ~\mathnormal{\mathbf{o}}~ w_{n-2^{l-1}}\left(t_{2^{l-1}+1},\ldots,t_n\right),  
\\
z_n(t_1,t_2,\ldots,t_n) &=& t_1~\mathnormal{\mathbf{o}}~ z_{n-1}(t_2,\ldots,t_n), 
\end{array}
\right.
\end{equation}
where $p=\ceil*{\frac{n}{2}}$, $q=\floor*{\frac{n}{2}}$, and $l=\ceil*{\log_2n}$. 
\\
\noindent
\textup{
The sequence $\{u_n\}$ follows the pattern of integers $\alpha_n=n-1$. It can also be labelled \textbf{FIFO (First In, First Out)} in that the arguments $t_1,t_2,\cdots,t_n$ are composed by $\mathnormal{\mathbf{o}}$ in increasing order of indices. The sequence $\{z_n\}$ on the other hand can be labelled \textbf{LIFO (Last In First Out)} and follows the pattern of integers $\alpha_n=1$. The sequence $\{v_n\}$ follows the pattern of integers $\alpha_n=\ceil*{\frac{n}{2}}$ and is labelled \textbf{AISO (All In, Split Out)} while the sequence $\{w_n\}$ follows the pattern of integers $\alpha_n=2^{\ceil*{\log_2n}-1}$, with the arguments $t_i$ split where the indice $i$ is the highest power of $2$. 
}

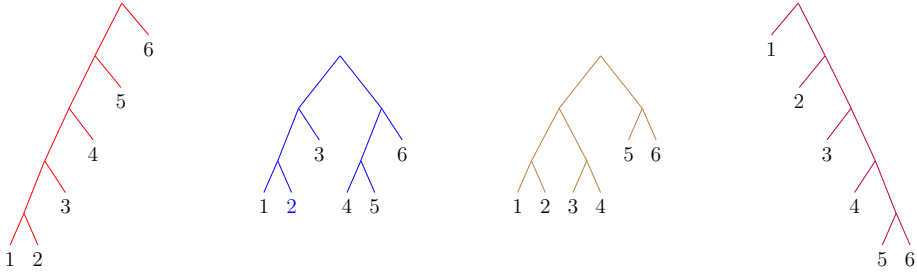
\begin{figure}[h]
\centering
\begin{minipage}{.2\textwidth}
\begin{tikzpicture}[scale=0.66]
{\color{red}
\Tree [ [ [ [ [ {\color{black} 1} {\color{black} 2} ] {\color{black} 3} ] {\color{black} 4} ] {\color{black} 5} ] {\color{black} 6} ]}
\end{tikzpicture}
\end{minipage}
\begin{minipage}{.2\textwidth}
\begin{tikzpicture}[scale=0.66]
{\color{blue}
\Tree [ [ [ {\color{black} 1} {\color{blue} 2} ] {\color{black} 3} ] [ [ {\color{black} 4} {\color{black} 5} ] {\color{black} 6} ] ]}
\end{tikzpicture}
\end{minipage}
\begin{minipage}{.2\textwidth}
\begin{tikzpicture}[scale=0.66]
{\color{brown}
\Tree  [ [ [ {\color{black} 1} {\color{black} 2} ] [ {\color{black} 3} {\color{black} 4} ] ] [ {\color{black} 5} {\color{black} 6} ] ]}
\end{tikzpicture}
\end{minipage}
\begin{minipage}{.2\textwidth}
\begin{tikzpicture}[scale=0.66]
{\color{purple}
\Tree [ {\color{black} 1} [ {\color{black} 2} [ {\color{black} 3} [ {\color{black} 4} [ {\color{black} 5} {\color{black} 6} ] ] ] ] ]}
\end{tikzpicture}
\end{minipage}
\caption{Binary trees for $u_6$, $v_6$, $w_6$ and $z_6$, from left to right.}
\label{onef}
\end{figure}
\end{example}
\noindent
In general, for an infinite sequence of points of real numbers, one can define $\mathnormal{\mathbf{o}}$-series as a way of composing successively the terms following a given pattern. More precisely, we have the following definition:

\begin{definition}\textup{($\om$-series)
Consider a sequence $\{t_n\}_{n \geq 0}$ of points in $I_a$ and an operation $\mathnormal{\mathbf{o}}: I_a \times I_a \to [0,\infty)$. The terms of the sequence can be composed successively starting from $t_0$ via the operation $\mathnormal{\mathbf{o}}$ as follows:
\begin{equation}
\left\{
\begin{array}{lll}
\omega^0&=& t_0 \\
\omega^1 &=& \mathnormal{\mathbf{o}}(t_0,t_1) \\
\omega^n &=& h_{n+1}(t_0,t_1,t_2,\ldots,t_n), ~~ h_{n+1} \in \Omega_n, ~~n \geq 2.
\end{array}
\right.
\end{equation}
The sequence $\{\omega^n\}_{n \geq 0}$ is defined to be the \emph{sequence of partial compositions of} $\{t_n\}_{n \geq 0}$ following the pattern of functions $\{h_n\}_{n \in \mathbb{N}}$ and is denoted $\omega^n= \som{0}{n} t_i.$ If $\{\omega^n\}_{n \in \mathbb{N}}$ converges, we say that $\{t_n\}_{n \geq 0}$ is \emph{composable} and denote $\lim_{n \to \infty} \omega^n$ by  $\som{0}{\infty} t_i$. 
\\
In general, the expression $\som{0}{\infty} t_i$ is called an infinite \emph{generalized series} (or \emph{$\mathnormal{\mathbf{o}}$-series})  \emph{following the  pattern of functions} $\{h_n\}_{n \in \mathbb{N}}$ and $\{t_n\}_{n \geq 0}$ is called the \emph{sequence of terms}, whether it is composable or not. If  $\{t_n\}_{n \geq 0}$ is composable, the generalized series  $\som{1}{\infty} t_i$ is said to converge; else, it is said to diverge.}
\end{definition}
\noindent
Given a sequence $\{t_n\}_{n \geq 0}$ of points in $I_a$, the $n$-th term $\omega^n$ of the sequence of partial compositions of $\{t_n\}_{n \geq 0}$ can be computed up to $C_n$ ways ($C_n$ being the $n$-th Catalan number). The following example serves as illustration:

\begin{example}
\textup{
Consider $I_0=[0,\infty)$ and $\om$ as in Example \ref{detailex}. If $t_n=n$ for any $n \geq 0$, then $\som{0}{3}t_i$ can be computed in $C_3=5$ ways:
\begin{equation*}
\som{0}{3}t_i=\som{0}{3}i=\Delta_3^j(0,1,2,3)=\left\{
\begin{array}{lll}
34, & \mbox{ for } j=1
\\ 
22, & \mbox{ for } j=2 
\\
18, & \mbox{ for } j=3
\\
16, & \mbox{ for } j=4
\\
12, & \mbox{ for } j=5.
\end{array}
\right.
\end{equation*}
}
\end{example}
  \noindent
It is therefore necessary to specify the pattern followed by the function $h_{n+1} \in \Omega_n$ (or $\Delta_n$) used to compute the partial composition $\som{0}{n} t_i$. To this end, we adopt the following definition:
\begin{definition}
Let $\{\alpha_n\}_{n \in \mathbb{N}}$ be a sequence of integers such that $1 \leq \alpha_n \leq n-1$ for all $n \in \mathbb{N}$.
If an $\om$-series $\som{0}{\infty} t_i$ following the pattern of functions $\{h_n\}_{n \in \mathbb{N}}$ is such that  $\{h_n\}_{n \in \mathbb{N}}$ follows the pattern of integers $\{\alpha_n\}_{n \in \mathbb{N}}$ in the sense of Definition \ref{pattern1},  $\som{0}{\infty} t_i$ is also said to follow the pattern of integers $\{\alpha_n\}_{n \in \mathbb{N}}$. 
\end{definition}
\noindent
Patterns are not necessary when the operation is $\mathnormal{\mathbf{o}}$ is associative as illustrated in the following example.
\begin{example}
\textup{
Let $\{\omega^n\}_{n \geq 0}$ be the sequence of partial compositions of a sequence $\{t_n\}_{n \geq 0}$ of real numbers in the interval $[a,\infty)$, with $\mathnormal{\mathbf{o}}:[a,\infty) \times [a,\infty) \to [a,\infty)$ a function. 
\begin{enumerate}
\item If $a=0$ and $\mathnormal{\mathbf{o}}(u,v)=u+v$ for all $u,v \geq 0$, then whatever the pattern followed, $\mathnormal{\mathbf{o}}$-series are series in the usual sense, i.e., $\som{0}{n} t_i=\displaystyle \sum_{i=0}^n t_i$ for $n \geq 0$;
\item 
If $\mathnormal{\mathbf{o}}(u,v)=\max\{u,v\}$ for all $u,v \geq a$, the sequence $\{\omega^n\}_{n \geq 0}$ of partial compositions of $\{t_n\}_{n \in \geq 0}$ following any pattern of functions is such that
$\omega^n=\som{0}{n} t_i =\max\{t_0,t_2,\ldots,t_n\}$, which means that $\som{0}{\infty} t_i=\lim_{n \to \infty} t_n$ if the sequence of terms $\{t_n\}_{n \in \mathbb{N}}$ is nondecreasing and  $\som{0}{\infty} t_i=t_0$ if $\{t_n\}_{n \in \mathbb{N}}$ is non-increasing.
\item If $a=1$ and $\mathnormal{\mathbf{o}}(u,v)=uv$ for all $u, v \geq 1$, then $\som{0}{n} t_i =\displaystyle \prod_{i=0}^n t_i$ for $n \geq 0$.
\end{enumerate}
}
\end{example}
\noindent
In general, when $\om$ satisfies conditions $(E_1)$ and $(E_2)$ with a function $\lambda$, then $\om$ is associative and given a sequence $\{t_n\}$ of numbers greater than or equal to $a$, we have for all $n \geq 0$,
\begin{equation}\label{bonsoir}
\som{0}{n} t_i=\lambda^{-1} \left(\displaystyle \sum_{i=0}^n \lambda(t_i) \right).
\end{equation}
\\
\\
In the next subsection, we consider a case when $\om$ is not associative.

\subsection{Polygon $\om$-inequalities in b-metric spaces}

Motivated by the relaxed triangle inequality in a b-metric space $(X,d,s)$, $s\geq 1$, which is a $0$-upward O-metric space, we consider the case where $a=0$ and $\mathnormal{\mathbf{o}}$ is defined for pairs $(u,v)$ of elements in $[0,\infty)$ by $\mathnormal{\mathbf{o}}(u,v)=s(u+v)$.  
\\
\\
Let $\{\omega^n\}_{n \geq 0}$ be the sequence of partial compositions of a sequence $\{t_n\}_{n \geq 0} \subset [0,\infty)$, with $\om:[0,\infty) \times [0,\infty) \to [0,\infty)$ defined by $\om(u,v)=s(u+v)$ for some $s >0$. 
\\
\\
If the $\om$-series follows the pattern of integer $\{1\}_{n \in \mathbb{N}}$, then for $n \geq 2$, $\som{0}{n} t_i=z_{n+1}(t_0,t_1,\ldots,t_n)$, where $\{z_n\}$ is defined as in (\ref{recursion}):
\begin{equation}\label{zn}
\begin{array}{lcl}
\som{0}{n} t_i=z_{n+1}(t_0,t_1,\ldots,t_n) &=& s[t_0+z_n(t_1,\ldots,t_{n})]
\\
&=& st_0 + sz_n(t_1,\ldots,t_{n})
\\
&=& st_0+ s^2[t_1+z_{n-1}(t_2,\ldots,t_n)]
\\
&=& st_0+ s^2t_1 + s^2 z_{n-1}(t_2,\ldots,t_n)
\\
&\vdots&
\\
&=&\displaystyle \sum_{i=1}^{n-1}s^i t_{i-1}+s^{n-1}z_2(t_{n-1},t_{n})
\\
& = & \displaystyle \sum_{i=1}^{n-1}s^i t_{i-1}+ s^nt_{n-1}+s^n t_{n}
\\
&=&\displaystyle \sum_{i=1}^{n}s^i t_{i-1}+ s^n t_{n}.
\end{array}
\end{equation}
Therefore, for any $n \in \mathbb{N}$, $\som{0}{n} t_i=z_{n+1}(t_0,t_1,\ldots,t_n)=\displaystyle \sum_{i=1}^{n}s^i t_{i-1}+s^n t_n.$
\\
\\
Similarly, if the $\om$-series follows the pattern of integers $\{n-1\}_{n \in \mathbb{N}}$, then for $n \in \mathbb{N}$,
\begin{equation}\label{Godshelp}
\som{0}{n} t_i=u_{n+1}(t_0,t_1,\ldots,t_n)=s^{n}t_0+\sum_{i=1}^{n}s^i t_{n-i+1},
\end{equation}
where $\{u_n\}$ is defined as in (\ref{recursion}).
\\%
\\%
Now, consider the $\om$-series following the pattern of integers $\{2^{\ceil*{\log_2 n}-1}\}_{n \in \mathbb{N}}$. Then
$\som{1}{n} t_i=w_{n}(t_1,t_2,\cdots,t_n)$ for all $n \in \mathbb{N}$, where $\{w_n\}$ is defined as in (\ref{recursion}).
\\
If $l=l(n)=\ceil*{\log_2(n)}$, with $n \geq 2$, then $1 \leq 2^{l-1}<n \leq 2^l$ and
\begin{equation}\label{GodisGood}
w_{n}(t_1, t_2, \ldots, t_{n}) =s\left(w_{2^{l-1}}\left(t_1,\ldots,t_{2^{l-1}}\right)+w_{n-2^{l-1}}\left(t_{2^{l-1}+1},\ldots,t_{n}\right) \right).
\end{equation}
\noindent
By a simple recursion, for any $r \geq 1$, 
\begin{equation}\label{GodGood}
\begin{array}{lcl} w_{2^r}(t_1,t_2,\ldots,t_{2^r}) &=& s[w_{2^{r-1}}(t_1,t_2,\ldots,t_{2^{r-1}})+w_{2^{r-1}}(t_{2^{r-1}+1},t_{2^{r-1}+2},\ldots,t_{2^{r}})]
\\
&=& s^2\left[w_{2^{r-2}}(t_1,t_2,\ldots,t_{2^{r-2}})+ w_{2^{r-2}}(t_{2^{r-2}+1},t_{2^{r-2}+2},\ldots,t_{2^{r-1}})\right.
\\
&&
+\left. w_{2^{r-2}}(t_{2^{r-1}+1},\ldots,t_{3 \times 2^{r-2}})+ w_{2^{r-2}}(t_{3 \times 2^{r-2}+1},\ldots,t_{2^{r}}) \right]
\\
&\vdots&
\\
&=&\displaystyle s^r\sum_{i=1}^{2^r}w_1(t_i)=s^r\sum_{i=1}^{2^r}t_i. \end{array}
\end{equation}
\noindent
Thus 
$\left\{\begin{array}{lll}
w_{2^{l-1}}\left(t_1,\ldots,t_{2^{l-1}}\right) = \displaystyle  s^{l-1}\sum_{i=1}^{2^{l-1}}t_i  \\
w_{n-2^{l-1}}\left(t_{2^{l-1}+1},\ldots,t_{n}\right) \leq w_{2^{l-1}}\left(t_{2^{l-1}+1},\ldots,t_{n},0,0,\ldots,0\right) = \displaystyle s^{l-1}\sum_{i=2^{l-1}+1}^{n}t_i. 
\end{array}\right.$ 
\\
Hence, for $n \geq 2$,
\begin{equation}\label{wn}
w_{n}(t_1, t_2, \ldots, t_{n}) \leq s\left[s^{l-1}\sum_{i=1}^{2^{l-1}}t_i+s^{l-1}\sum_{i=2^{l-1}+1}^{n}t_i\right]=s^l\sum_{i=1}^{n} t_i.
\end{equation}
In fact, repeating the processes in (\ref{GodisGood}) and (\ref{GodGood}), we have for $n$ sufficiently large,
\begin{equation}\label{praiz}
\begin{array}{lcl}
w_{n}(t_1, t_2, \ldots, t_{n}) &=&s\left(w_{2^{l(n)-1}}\left(t_1,\ldots,t_{2^{l(n)-1}}\right)+w_{n-2^{l(n)-1}}\left(t_{2^{l(n)-1}+1},\ldots,t_{n}\right) \right)
\\
&=& \displaystyle s\left(s^{l(n)-1}\sum_{i=1}^{2^{l(n)-1}}t_i  +w_{n-2^{l(n)-1}}(t_{2^{l(n)-1}+1},\ldots,t_n) \right)
\\
&=& \displaystyle s\left(s^{l_0-1}\sum_{i=1}^{{n_1}}t_i  +w_{n-n_1}(t_{n_1+1},\ldots,t_n) \right), ~~l_0=l(n), ~ n_1=2^{l_0-1},
\\
&=& \displaystyle s^{l_0}\sum_{i=1}^{{n_1}}t_i  +sw_{n-n_1}(t_{n_1+1},\ldots,t_n)
\\
&=&\displaystyle s^{l_0}\sum_{i=1}^{n_1}t_i +s\left[ s^{l_1}\sum_{i=n_1+1}^{n_2}t_i  +sw_{n-n_2}(t_{n_2+1},\ldots,t_n) \right]
\\
&& \mbox{ where }  l_1=l(n-n_1), ~ n_2=n_1+2^{l_1-1},
\\
&=&\displaystyle s^{l_0}\sum_{i=1}^{n_1}t_i +s^{1+l_1}\sum_{i=n_1+1}^{n_2}t_i  +s^2w_{n-n_2}(t_{n_2+1},\ldots,t_n) 
\\
&\vdots&
\\
&=& \displaystyle s^{l_0}\sum_{i=1}^{n_1}t_i +s^{1+l_1}\sum_{i=n_1+1}^{n_2}t_i  +\cdots+ s^{r-1+l_{r-1}}\sum_{i=n_{r-1}+1}^{n_{r}}t_i
\\&& +s^r w_{n-n_r}(t_{n_r+1},\ldots,t_n), 
\end{array}
\end{equation}
where $r \in \mathbb{N}$ is such that $n-n_r \geq 1$, and $l_j=\ceil*{\log_2(n-n_j)}$ and $n_{j+1}=n_{j}+2^{l_j-1}$ for $j \in [0,r]$, with $n_0=0$.
\\
The sequence $(n_j)$ of integers is strictly increasing and bounded above by $n$ hence finite: there is $N\in \mathbb{N}$ such that $N-1=\max\{r \in \mathbb{N}: ~ n-n_r \geq 2\}$. By definition of $N$, $n-n_{N}<2$. Suppose $n=n_N$. Then $l_{N-1}=\ceil*{\log_2(n-n_{N-1})}=\ceil*{\log_2(n_N-n_{N-1})}=l_{N-1}-1$, a contradiction. Thus $n-n_N=1$ so $n_N=n-1$. Therefore, from (\ref{praiz}), taking $r=N$, the last term of the inequality becomes $s^Nw_1(t_n)=s^N t_n$ so that:
\begin{equation}\label{ghen}
\begin{array}{lcl}
\som{1}{n} t_i &=&w_{n}(t_1,t_2,\cdots,t_n)
= \displaystyle \sum_{r=0}^{N-1}\left( s^{r+l_r}\sum_{i=n_r+1}^{n_{r+1}}t_i \right)+s^N t_n
\\
&=& \displaystyle \sum_{r=0}^{N}\left( s^{r+l_r}\sum_{i=n_r+1}^{n_{r+1}}t_i \right),
\end{array}
\end{equation}
where $(n_r)_{0 \leq r \leq N+1}$ and $(l_r)_{0 \leq r \leq N}$ are such that:
\begin{equation}\label{Godmercy}
\left\{
\begin{array}{lcl}
n_0 &=& 0
\\
n_{j+1} &=& n_j+2^{l_j-1}=\displaystyle \sum_{r=0}^j2^{l_r-1}, ~~ j \in \{0,1,\ldots,N-1\}
\\
N-1&=&\max\{r \in \mathbb{N}: ~ n-n_r \geq 2\}
\\
l_j&=&\ceil*{\log_2(n-n_j)}, ~~ j \in \{0,1,\ldots,N\}
\\
n_{N+1}&=& n_N +1 = n.
\end{array}
\right.
\end{equation}
The sequence $(n_j)$ in (\ref{Godmercy}) provides the binary representation of $n \in \mathbb{N}$. Indeed, from (\ref{Godmercy}),
$$n=n_N +1=  \displaystyle \sum_{r=0}^{N-1} 2^{l_r-1}+1=2^{l_0-1}+2^{l_1-1}+\cdots+2^{l_{N-1}-1}+1.$$
Therefore, $n=\displaystyle \sum_{j=0}^{k} a_j 2^j$, where $k=l_0-1$, and 
$a_j =\left\{ \begin{array}{lll}  1, &\mbox{ if } j\in\{0, l_0 -1, l_1-1, \cdots, l_{N-1}-1\} \\ 0, &\mbox{ elsewhere.} \end{array}  \right.$
\\
\\
As application, we determine some polygon (relaxed) inequalities that hold in b-metric spaces in the following proposition:
\begin{proposition}[Polygon inequalities in a b-metric space]
Let $(X,d,s)$ be a b-metric space, with $s \geq 1$. Then given $n+2$ points $x_0,x_1,\cdots,x_{n+1}$, where $n \in \mathbb{N}$, the following polygon inequalities hold:
\begin{equation}\label{nameit}
d(x_0,x_{n+1}) \leq \sum_{i=1}^{n+1} a_i d(x_{i-1},x_i)
\end{equation}
with $(a_i)_{1 \leq i \leq n+1}$ a sequence of nonnegative real numbers such that either of the following hold:
\begin{enumerate}
\item\label{checke} There exist $p \neq q \in \{1, 2, \ldots, n+1\}$ such that $a_p=a_q=s^n$, with the other $a_i's$ distinct and equal to some power $s^r$ of $s$, with $1 \leq r <n-1$.
\item\label{checkeq} The $a_i$'s are constant, $a_i=K$, where $K=\dfrac{1}{n+1}\left(\dfrac{2s^{n+1}-s^n-s}{s-1} \right)$ or $K=s^{\ceil*{\log_2(n+1)}}$.
\end{enumerate}
\end{proposition}
\begin{proof}
Under the conditions of the proposition, the polygon inequality (\ref{comp11}) holds for any $\Delta_n \in \Omega_n:$
$$d_{\mathnormal{\mathbf{o}}}(x_0,x_{n+1}) \leq \Delta_n\left(d_{\mathnormal{\mathbf{o}}}(x_0,x_1),d_{\mathnormal{\mathbf{o}}}(x_1,x_2),\ldots,d_{\mathnormal{\mathbf{o}}}(x_{n},x_{n+1})\right).$$
If we let $\Delta_n=z_{n+1}$, where $(z_n)$ is the recursion defined in (\ref{recursion}), then from (\ref{zn}), we obtain:
\begin{equation}\label{mercy1}
d(x_0,x_{n+1}) \leq \sum_{i=1}^{n}s^id(x_{i-1},x_i)+ s^nd(x_{n},x_{n+1}).
\end{equation}
If we let $\Delta_n=u_{n+1}$, where $(u_n)$ is the recursion defined in (\ref{recursion}), then from (\ref{Godshelp}), we obtain:
\begin{equation}\label{mercy2}
d(x_0,x_{n+1}) \leq s^nd(x_{0},x_{1})+\sum_{i=1}^{n}s^i d(x_{n-i+1},x_{n-i+2}).
\end{equation}
Rearranging the points $x_i$, $1 \leq i \leq n$, we obtain $d(x_0,x_{n+1}) \leq \displaystyle \sum_{i=1}^{n+1} a_i d(x_{i-1},x_i)$, where $(a_i)_{1 \leq i \leq n+1}$  satisfies condition \ref{checke}. 
\\
In fact, if one sums $n+1$ of such inequalities $d(x_0,x_{n+1}) \leq \displaystyle \sum_{i=1}^{n+1} a_{\sigma_j(i)} d(x_{i-1},x_i)$, $1 \leq j \leq n+1$, with permutations  $\sigma_j \in S_{n+1}$ chosen such that for each $i$, the $\sigma_j(i)$ are distinct, then:
$$(n+1)d(x_0,x_{n+1}) \leq \left(\sum_{i=1}^{n}s^i +s^n \right)\sum_{i=1}^{n+1} d(x_{i-1},x_i).$$
Therefore, 
\begin{equation}\label{finalite}
\left\{
\begin{array}{lll}
\displaystyle d(x_0,x_{n+1}) \leq \sum_{i=1}^{n+1} K d(x_{i-1},x_i), \mbox{ where:}
\\
K=\displaystyle \dfrac{1}{n+1}\left(\sum_{i=1}^{n}s^i +s^n \right)=\dfrac{1}{n+1}\left(\dfrac{2s^{n+1}-s^n-s}{s-1} \right).
\end{array}
\right.
\end{equation}
If in (\ref{comp11}) we let $\Delta_n=w_{n+1}$, where $(w_n)$ is the recursion defined in (\ref{recursion}), then from (\ref{wn}), we obtain:
\begin{equation}\label{finality}
d(x_0,x_{n+1}) \leq \sum_{i=1}^{n+1} s^{\ceil*{\log_2(n+1)}} d(x_{i-1},x_i).
\end{equation}
From (\ref{finalite}) and (\ref{finality}), condition \ref{checkeq} holds.
\end{proof}
\noindent
It should be noted that polygon inequalities are used to prove that ``contractive" sequences are Cauchy sequences. In a b-metric space $(X,d,s)$, a sequence $\{x_n\}_{n \geq 0}$ is said to be contractive if $d(x_{n},x_{n+1}) \leq kd(x_{n-1},x_n)$ for all $n \in \mathbb{N}$, where $k \in [0,1)$.  Suzuki \cite{suzuki} combined inequalities of type (\ref{nameit}) to show that contractive sequences in b-metric spaces are Cauchy sequences. In the next section, we consider contractive sequences in the general context of an O-metric space.

\section{Contractions in O-metric spaces}


\begin{definition}\label{contraction}
\textup{
Let $(X,d_\om,a)$ be an $a$-upward O-metric space, with $\om$ nondecreasing in both variables, continuous at $(a,a)$ and $\om(a,a)=a$. 
 Let $\varphi:[0,\infty) \times [a,\infty) \to [a,\infty)$ be a function satisfying the following conditions:
\begin{itemize}
\item[$(\varphi_1)$] $\varphi(0,t)=\varphi(r,a)=a$ for all $t \geq a$ and $r \geq 0$;
\item[$(\varphi_2)$] $\varphi\big|_{(0,\infty) \times (a,\infty)}$ is increasing on both variables, and continuous in the second variable at $a$; 
\item[$(\varphi_3)$] $\forall r_1,r_2 \in [0,\infty)$ $\forall t \in [a,\infty)$,  $\varphi(r_1,\varphi(r_2,t)) =\varphi(r_1r_2,t)$.
\end{itemize}
 A map $T:X \to X$ is said to be $k$-$\varphi$ \emph{Lipschitz} on $X$, with $k \geq 0$, if
\begin{equation}\label{Lips}
d_\om(Tx,Ty) \leq \varphi(k,d_\om(x,y)) ~~~ \forall x,y \in X.
\end{equation}
\noindent
The $k$-$\varphi$ Lipschitz map $T:X \to X$ is said to be a \emph{contraction} if $k <1.$ 
\\
A sequence $\{x_n\}_{n \geq 0}$ of points in $X$ is said to be a $k$-$\varphi$ \emph{contractive} sequence if  for all $n \in \mathbb{N}$,
\begin{equation}\label{SeqCo}
d_\om(x_n,x_{n+1}) \leq \varphi(k,d_\om(x_{n-1},x_n)).
\end{equation}
}
\end{definition}
\noindent
The following are examples of mappings satisfying conditions $(\varphi_1)-(\varphi_3)$.
\begin{example}
When $a=0$, the following maps $\varphi:[0,\infty) \times [0,\infty) \to [0,\infty)$ satisfy conditions $(\varphi_1)-(\varphi_3)$:
\begin{enumerate} 
\item $\varphi(t,u)=tu$ for all $t,u \geq 0$. 
\item $\varphi(t,u)=(1+u)^t-1$ for all $t,u \geq 0$.
\item $\varphi(t,u)=\ln(1-t+te^u)$ for all $t,u \geq 0$. 
\end{enumerate}
\end{example}

\begin{example}
The map $\varphi:[0,\infty) \times [1,\infty) \to [1,\infty)$ defined by $\varphi(t,u)=u^t$ satisfy conditions $(\varphi_1)-(\varphi_3)$ when $a=1$.
\end{example}

\begin{example}\label{nowp}
Given an $a$-upward O-metric space $(X,d_\om,a)$, if $\lambda:[a,\infty) \to [0,\infty)$ is an increasing function satisfying $\lambda(a)=0$ and $\lambda(u ~\om ~v)=\lambda(u)+\lambda(v)$ (as in conditions $(E_1)$ and $(E_2)$ of Theorem \ref{metricequiv}), the map $\varphi:[0,\infty) \times [a,\infty) \to [a,\infty)$ defined by $\varphi(t,u)=\lambda^{-1}\left(t\lambda(u)\right)$ satisfies conditions $(\varphi_1)-(\varphi_3)$.
\end{example}
\noindent
We note the following about maps satisfying condition  (\ref{Lips}):
\begin{remark}\textup{
Let $(X,d_\om,a)$ be an $a$-upward O-metric space.
\begin{enumerate}
\item
The term \emph{contraction} is justified. Indeed, let $T$ be a contraction, with the inequality $d_\om(Tx,Ty) \leq \varphi(k,d_\om(x,y))$ for all $x,y \in X$ and some $k<1$; then  $T$ contracts the symmetric function $D_{r}$ defined by $D_{r}(x,y)=\varphi(r,d_\om(x,y))$ for some $r > 0$ and all $x,y \in X$. In fact, $D_{r}$ is an $a$-upward $\mathnormal{\mathbf{o}}$-metric if $\varphi(r,\om(t_1,t_2)) \leq \om(\varphi(r,t_1),\varphi(r,t_2))$ for all $t_1,t_2 \geq a$. 
\item As expected of \emph{Lipschitz} maps, a  $k$-$\varphi$ Lipschitz map $T:X \to X$ as defined in Definition \ref{contraction} is continuous.  Indeed, let $G$ be an open set in $X$ (for the topology $\mathcal{T}_{d_{\mathnormal{\mathbf{o}}}}$). To show that $T^{-1}(G)$ is an open set for the topology, we let $x \in T^{-1}(G)$. Since $Tx \in G$, then there is $r>0$ such that $B(Tx,r) \subset G$.  Choose $\delta>0$ such that $\varphi(k, a+\delta)-a=r$. For any $y \in B(x,\delta)$, since $d_{\mathnormal{\mathbf{o}}}(x,y)<a+\delta$, we have that $d_{\mathnormal{\mathbf{o}}}(Tx,Ty)-a \leq \varphi(k, d_{\mathnormal{\mathbf{o}}}(x,y))-a<\varphi(k,a+\delta)-a=r$. Thus $Ty \in B(Tx,r) \subset G$ and $y \in T^{-1}(G)$.
\item A $k$-$\varphi$ Lipschitz map $T:X \to X$ also preserves O-convergence (i.e., $T$ is sequentially continuous): if a sequence $\{x_n\}$ of points in $X$ is such that $x_n  \xrightarrow{\text{O}} x$, then, for each $n \in \mathbb{N}$, $a \leq d_\om(Tx_n,Tx) \leq \varphi(k,d_\om(x_n,x))$, hence, as $n \to \infty$, $Tx_n  \xrightarrow{\text{O}} Tx$.
\end{enumerate}
}
\end{remark}
\noindent
In order to determine values of $k$ for which a $k$-$\varphi$ contractive sequence is a Cauchy sequence, we introduce the set $C_\varphi$ as in the proposition below.

\begin{proposition}\label{contractionset}
\textup{
Let $\om$ be nondecreasing in both variables, continuous at $(a,a)$ and  such that $\om(a,a)=a$. For a function $\varphi:[0,\infty) \times [a,\infty) \to [a,\infty)$ satisfying $(\varphi_1) - (\varphi_3)$, define the set $C_\varphi$ by:
\begin{equation}\label{converted}
\left|
\begin{array}{lll}
C_\varphi=\left\{r \geq 0 ~ | ~ \forall \epsilon \geq a ~ ~ \lim_{n,i \to \infty}h_{n,i}(r,\epsilon) = a \right\}, \mbox{ where} \\
~~~~~~~~~~~~~ h_{n,i}(r,\epsilon)=h(\varphi(r^n,\epsilon),\ldots,\varphi(r^{n+i},\epsilon)) \mbox{ for some } h \in \Omega_i.    
\end{array}
\right.\footnote{ $\Omega_i$ is as defined in Lemma \ref{polygon}. One can also write $h_{n,i}(r,\epsilon)=\somj{0}{i} \varphi(r^{n+j},\epsilon)$.}
\end{equation}
}
Then  $C_\varphi$ is an interval such that $0 \in C_\varphi \subset [0,1)$.
\end{proposition}

\begin{proof}
If $r=0$, then $\varphi(r^n,\epsilon)=\ldots=\varphi(r^{n+i},\epsilon)=a$ for all $\epsilon>a$ and $n,i \in \mathbb{N}_0$, hence $h_{n,i}(r,\epsilon)=h(a,a,\ldots,a)=a$ for all $h \in \Omega_i$. Therefore $0 \in C_\varphi$.
\\
Suppose $r \in C_\varphi$ and $s \in [0,r]$. Since $\varphi$ is nondecreasing in the first variable, $\varphi(s^m,\epsilon) \leq \varphi(r^m,\epsilon)$ for $m \in \{n,n+1,\ldots,n+i\}$ with $n,i \in \mathbb{N}_0$. As $\om$ is nondecreasing in both variables, $h$ is nondecreasing in all variables hence $h_{n,i}(s,\epsilon) \leq h_{n,i}(r,\epsilon)$ for all $h \in \Omega_i$. Thus $s \in C_\varphi$ and $C_\varphi$ is an interval.
\\
Let $r=1$. For all $n,i \in \mathbb{N}_0$, $\epsilon>a$ and  for any $h \in \Omega_i$, $~h_{n,i}(1,\epsilon)=h(\varphi(1,\epsilon),\ldots,\varphi(1,\epsilon))$. $h_{n,i}(1,\epsilon) \not\to a$ for $\varphi(1,\epsilon)>a$ hence $1 \notin C_\varphi$.
\\
Therefore $C_\varphi$ is an interval, $0 \in C_\varphi \subset [0,1)$ and $\sup C_\varphi \leq 1$.
\end{proof}
\noindent
In the next lemmas, we find $C_\varphi$ for some maps $\varphi$ satisfying conditions $(\varphi_1)-(\varphi_3)$. 
\begin{lemma}\label{potentia1}
Suppose $\om$ is nondecreasing in both variables, continuous at $(a,a)$ with $\om(a,a)=a$, and satisfying conditions $(E_1)$ and $(E_2)$ of Theorem \ref{metricequiv} for a function $\lambda:[a,\infty) \to [0,\infty)$. If $\varphi:[0,\infty) \times [a,\infty) \to [a,\infty)$ is defined by $\varphi(t,u)=\lambda^{-1}\left(t\lambda(u)\right)$ $\forall t \geq 0 ~ \forall u \geq a$, then $C_\varphi=[0,1)$ \end{lemma}

\begin{proof}
It is easy to check that $\varphi$ so defined satisfies conditions $(\varphi_1)-(\varphi_3)$.  
Also, $\om$ is associative. Thus, if $r \geq 0$, $\epsilon>a$, $i \in \mathbb{N}_0$ and $h \in \Omega_i$, then by (\ref{bonsoir}), 
\begin{equation*}
h(t_0,t_1,\cdots,t_n)=\somj{0}{n} t_j=\lambda^{-1} \left(\displaystyle \sum_{j=0}^n \lambda(t_j) \right).
\end{equation*}
for $t_1,t_2,\ldots,t_{i+1} \geq a$. Therefore,
%
%
%
\begin{equation*}
\begin{array}{lcl}
\lambda(h_{n,i}(r,\epsilon))  &=& \lambda\left(h(\varphi(r^n,\epsilon),\varphi(r^{n+1},\epsilon),\ldots,\varphi(r^{n+i},\epsilon))\right)
\\
&=&\displaystyle \sum_{j=0}^{i} \lambda\left(\varphi(r^{n+j},\epsilon)  \right)
\\
&=&\displaystyle \sum_{j=0}^{i} r^{n+j}\lambda(\epsilon) 
\\
&=&\dfrac{1-r^{i+1}}{1-r}r^n\lambda(\epsilon)  \to 0 \mbox{ as $n,i \to \infty$ if and only if } r<1.
\end{array}
\end{equation*}
Thus $C_\varphi=[0,1)$ and $\sup C_\varphi=1$.
\end{proof}

\begin{lemma}\label{potentx}
If $\om$ is defined by $\om(u,v)=s(u+v)$ for all $u,v \geq 0$, where $s$ is a constant greater than or equal to $1$,
then $C_\varphi=[0,1)$ for $\varphi:[0,\infty) \times [0,\infty) \to [0,\infty)$ defined by $\varphi(t,u)=tu$ for all $t,u \geq 0$. 
\end{lemma}
\begin{proof}
The function $\om$ so defined is nondecreasing in both variables, continuous at $(0,0)$ and such that $\om(0,0)=0$. Furthermore, the map
$\varphi$ defined by $\varphi(t,u)=tu$ for all $t,u \geq 0$ satisfies conditions $(\varphi_1) - (\varphi_3)$ for $a=0$. Let $r \in (0,1)$, $\epsilon>0$, $n, i \in \mathbb{N}_0$. Let $l \in \mathbb{N}$ be such that $sr^{2^l} <1$.
\\
If $i +1 \leq 2^l$ then taking $h=w_{i+1} \in \Omega_i$, where $\{w_n\}$ is defined as in (\ref{recursion}), we have from (\ref{wn}) that:
\begin{equation}\label{mediate}
\begin{array}{lcl}
h_{n,i}(r,\epsilon) 
&=& w_{i+1}(r^n\epsilon,\ldots,r^{n+i}\epsilon)
\\
&\leq& s^{\ceil*{\log_2(i+1)}} \displaystyle  \sum_{j=0}^i  r^{n+j} \epsilon 
\\
&\leq& \displaystyle s^l r^n \sum_{j=0}^\infty r^j \epsilon =s^l r^n\dfrac{\epsilon}{1-r} \to 0 \mbox{ as } n,i \to \infty.
\end{array}
\end{equation}
If $2^l <i+1$, then putting $\mu=\floor*{\dfrac{i+1}{2^l}}$, we take $h \in \Omega_{i}$ defined for all $t_1,\ldots,t_{i+1} \geq 0$ by
\begin{equation}\label{justb}
h(t_1,t_2,\ldots,t_{i+1})=z_{\mu+1}\left(w_{2^l}(T_1),w_{2^l}(T_2),\ldots,w_{2^l}(T_\mu),w_{i+1-\mu 2^l}(t_{\mu 2^l+1},\ldots,t_{i+1})\right)
\end{equation}
where $T_j=(t_{(j-1)2^l+1},\ldots,t_{j2^l})$ for $1 \leq j \leq \mu$ and $z_{\mu+1} \in \Omega_\mu$ as defined in (\ref{recursion}). If we write
$R_j=(r^{n+(j-1)2^l}\epsilon,\ldots,r^{n+j2^l-1}\epsilon)$, then
$$
h_{n,i}(r,\epsilon)=z_{\mu+1}\left(w_{2^l}(R_1),w_{2^l}(R_2),\ldots,w_{2^l}(R_\mu),w_{i+1-\mu 2^l}(r^{n+\mu 2^l}\epsilon,\ldots,r^{n+i}\epsilon)\right).$$
From (\ref{mediate}), $w_{2^l}(R_j) \leq \dfrac{s^l r^{n+(j-1)2^l}\epsilon}{1-r}$ for all $j$. Since $\om$ is nondecreasing in both variables, $z_{\mu+1}$ is nondecreasing in all variables hence from (\ref{mediate}) and (\ref{zn}),
\begin{equation*}
\begin{array}{lcl}
h_{n,i}(r,\epsilon) & \leq & z_{\mu+1}\left(\dfrac{s^l r^{n}\epsilon}{1-r},\dfrac{s^l r^{n+2^l}\epsilon}{1-r},\ldots,\dfrac{s^l r^{n+(\mu-1)2^l}\epsilon}{1-r},w_{i+1-\mu 2^l}(r^{n+\mu 2^l}\epsilon,\ldots,r^{n+i}\epsilon)\right)
\\
& \leq &  z_{\mu+1}\left(\dfrac{s^l r^{n}\epsilon}{1-r},\dfrac{s^l r^{n+2^l}\epsilon}{1-r},\ldots,\dfrac{s^l r^{n+(\mu-1)2^l}\epsilon}{1-r},\dfrac{s^l r^{n+\mu 2^l}\epsilon}{1-r} \right)
\\
&=& \displaystyle\sum_{j=1}^{\mu} \left[\dfrac{s^{j+l} r^{n+(j-1)2^l} \epsilon}{1-r} \right] + \dfrac{s^{\mu+l} r^{n+\mu 2^l}\epsilon}{1-r}
\\
&\leq& \dfrac{r^n  s^l\epsilon}{1-r} \displaystyle \sum_{j=1}^{\mu+1} s^jr^{(j-1)2^l} \leq \dfrac{r^n s^l \epsilon}{1-r} \displaystyle \sum_{j=1}^{\mu+1} \left(sr^{2^l}\right)^j \to 0 \mbox{ as } n \to \infty.
\end{array}
\end{equation*}
Thus $r \in C_\varphi$ and so $C_\varphi=[0,1)$.
\end{proof}
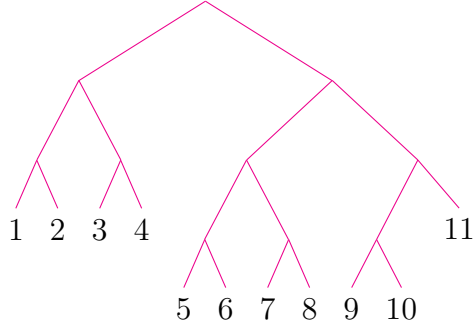
\begin{figure}[h]
\centering
\begin{tikzpicture}
{\color{magenta}
\Tree [                 [ [ {\color{black} 1} {\color{black} 2} ] [ {\color{black} 3} {\color{black} 4} ] ]                     [ [ [ {\color{black} 5} {\color{black} 6} ] [ {\color{black} 7} {\color{black} 8} ] ] [  [ {\color{black} 9} {\color{black} 10} ] {\color{black} 11} ] ]  ] }
\end{tikzpicture}
\caption{Binary tree for $h$ defined in (\ref{justb}) for $i=10$, $l=2$, $\mu=2$.}
\label
{twof}
\end{figure}

\section{Fixed point theorems in O-metric spaces}

In this section, we prove some fixed point theorems in the setting of upward O-metric spaces satisfying conditions $(U_1)$ and $(U_2)$, beginning with the contraction principle. The proof rely on the polygon inequalities and the use of contractive sequences. 

\subsection{The Banach contraction principle}

\begin{theorem}{(Banach Contraction Principle)}\label{most}
Let $(X,d_\om,a)$ be a complete $a$-upward O-metric space, with 
$\om$ nondecreasing in both variables, continuous at $(a,a)$, and $\om(a,a)=a$. Let $T:X \to X$ be a $k$-$\varphi$ contraction, where $\varphi:[0,\infty) \times [a,\infty) \to [a,\infty)$ satisfies conditions $(\varphi_1) - (\varphi_3)$, and $k<\kappa$, with $\kappa:=\sup C_\varphi$. 
Then $T$ has a unique fixed point. 
\end{theorem}
\begin{proof} Let $x_0 \in X$ and $\{x_n\}_{n \in \mathbb{N}}$ be a sequence\footnote{As seen in (\ref{cry}), the sequence $\{x_n\}$ of iterates of $T$ is a contractive sequence.} such that $x_{n+1}=Tx_n$ for $n \geq 0$. For
 $n \geq 0$,
\begin{equation}\label{cry}
\begin{array}{lcl}
d_\om(x_n,x_{n+1}) = d_\om(Tx_{n-1},Tx_n) & \leq & \varphi(k,d_\om(x_{n-1},x_n))
\\
&\leq& \varphi(k, \varphi(k,d_\om(x_{n-2},x_{n-1})))
\\
&=& \varphi(k^2,d_\om(x_{n-2},x_{n-1}))
\\
&\vdots&
\\
&\leq& \varphi(k^n,d_\om(x_0,x_1)). 
\end{array}
\end{equation}
\noindent
From Lemma \ref{polygon}, $d_\om(x_n,x_{n+i}) \leq h(d(x_n,x_{n+1}),d(x_{n+1},x_{n+2}),\ldots,d(x_{n+i-1},x_{n+i}))$ for all $h \in \Omega_{i-1}$, $n, i \in \mathbb{N}$. Since $h$ is nondecreasing in all its variables, and for all $n \in \mathbb{N}$, $d_\om(x_n,x_{n+1}) \leq \varphi(k^n,d_\om(x_0,x_1))$, we have: $$d_\om(x_n,x_{n+i}) \leq  h_{n,i-1}(k,d_\om(x_0,x_1)) ~~\forall n, i \in \mathbb{N}.$$
If $d_\om(x_0,x_1)=a$, then $x_0=x_1=Tx_0$ and $x_0$ is a fixed point of $T$.\\
Suppose that $d_\om(x_0,x_1)>a$. Since $k<\kappa$, $d_\om(x_n,x_{m}) \to a$  as $n,m \to \infty$, hence $\{x_n\}$ is a Cauchy sequence\footnote{ $C_{\varphi}$ can thus be considered as the \emph{interval of Cauchyness of contractive sequences}: a $k$-$\varphi$ contractive sequence is a Cauchy sequence if $k \in [0, \sup C_\varphi)$.}, and thus converges to some point $x^* \in X$.\\
For all $n \in \mathbb{N}$, $d_\om(x_{n+1},Tx^*) =d_\om(Tx_n,Tx^*) \leq \varphi(k,d_\om(x_n,x^*)) \to \varphi(k,a)=a$. Hence $x_{n+1} \xrightarrow{\text{O}} Tx^*$ and $Tx^*=x^*$ since the O-limit is unique.\\
Suppose 
that  $x_1^*$ and $x_2^*$ are two fixed points of $T$ such that $x_1^*\neq x_2^*$.  Then
$d_\om(x_1^*,x_2^*) =d_\om(Tx_1^*,Tx_2^*) \leq \varphi(k,d_\om(x_1^*,x_2^*)) < \varphi(1,d_\om(x_1^*,x_2^*))= d_\om(x_1^*,x_2^*)$, a contradiction. Hence, the fixed point of $T$ is unique.
\end{proof}
\begin{corollary}\label{generalE}
Let $(X,d_\om,a)$ be an $a$-upward O-metric space such that there is a function $\lambda:[a,\infty) \to [0,\infty)$ satisfying conditions $(E_1)$ and $(E_2)$. Any map $T:X \to X$ such that for some $k \in (0,1)$ and for all $x,y \in X$
\begin{equation*}
d_\om(Tx,Ty) \leq \lambda^{-1}\left(k \lambda \left(d_\om(x,y)\right)\right),
\end{equation*}
has a unique fixed point in $X$.
\end{corollary}
\begin{proof}
Under the conditions of the corollary, $\kappa=\sup C_\varphi=1$ for $\varphi(r,u)=\lambda^{-1}(r \lambda(u))$ as seen in Proposition \ref{potentia1}. The map $T$ is a $k$-$\varphi$ contraction, with $k<1=\kappa$. 
 Thus, from Theorem \ref{most}, $T$ has a unique fixed point in $X$.
\end{proof}


\begin{corollary}
Let $(X,d,s)$ be a b-metric space, with $s \geq 1$ and $T:X \to X$ a mapping such that for some $k \in \left(0,1\right)$ and for all $x,y \in X$
\begin{equation*}
d(Tx,Ty) \leq kd(x,y).
\end{equation*}
Then $T$ has a unique fixed point in $X$.
\end{corollary}
\begin{proof} The map $T$ is a $k$-$\varphi$ contraction with $\varphi:[0,\infty) \times [0,\infty) \to [0,\infty)$ defined by $\varphi(r,u)=ru$ for all $r,u \geq 0$. From Lemma \ref{potentx}, $\kappa=\sup C_\varphi=1$ hence from Theorem \ref{most}, $T$ has a unique fixed point. 
\end{proof}
\noindent
Next, we consider generalized contractions on O-metric spaces.

\subsection{Generalized contractions}

\begin{definition}
\label{onlyyou}
Let $(X,d_\om,a)$ be an $a$-upward O-metric space, and $\alpha:X \times X \to [0,\infty)$ a map.
Denote by $\Psi$ the family of nondecreasing functions
$\psi: [a,\infty) \to [a,\infty)$ such that 
\begin{equation}\label{complique}
\lim_{n,i \to \infty} \somj{n}{n+i} \psi^{(j)}(\epsilon)=a ~~~ \forall \epsilon>a,
\end{equation}
with the $\om$-series in (\ref{complique}) following some pattern of functions 
$\{h_i\}_{i \in \mathbb{N}}$ (where $h_{i+1} \in \Omega_i$ $\forall i \geq 0$).
\\
A map $T:X \to X$ is said to be an $\alpha$-$\psi$ contractive mapping, where $\psi \in \Psi$, if
\begin{equation}
\alpha(x,y)d_\om(Tx,Ty) \leq \psi(d_\om(x,y)), ~~ \forall x,y \in X.
\end{equation}
\end{definition}
\begin{remark}\label{escapee}\textup{
It should be noted that (\ref{complique}) can written as:
\begin{equation}\label{nahnah}
\forall \epsilon>a ~\lim_{n,i \to \infty} h\left(\psi^{(n)}(\epsilon),\psi^{(n+1)}(\epsilon),\ldots,\psi^{(n+i)}(\epsilon)\right)=a, \mbox{ where } h \in \Omega_i.
\end{equation}
}
\end{remark}
\noindent
The following remark states the relationship between $k$-$\varphi$ contractions (in Definition \ref{contraction}) and $\alpha$-$\psi$ contractions (in Definition \ref{onlyyou}).
\begin{remark}\label{sleep}
\textup{
Now, let $\psi: [a, \infty) \to [a, \infty)$ be defined for $t\geq a$ by 
\begin{equation}
\psi(t)=\varphi(k,t),
\end{equation}
where $\varphi$ satisfies conditions $(\varphi_1)-(\varphi_3)$. The map $\varphi$ so defined is nondecreasing and for each $n \in \mathbb{N}$ and $t \in [a, \infty)$, $\psi^{(n)}(t)=\varphi(k^n,t)$. If $k<\sup C_\varphi$, then $k \in C_\varphi$ hence, from (\ref{converted}), for any $\epsilon\geq a$, $\lim_{n,i \to \infty} h\left(\psi^{(n)}(\epsilon),\psi^{(n+1)}(\epsilon),\ldots,\psi^{(n+i)}(\epsilon)\right)=a$ where $h \in \Omega_i$. Conditions (\ref{nahnah}) and (\ref{complique}) therefore hold, and $\psi \in \Psi$. Thus, if $k \in [0, \sup  C_\varphi)$,  a $k$-$\varphi$ contraction $T:X \to X$ is an $\mathbf{1}$-$\psi$ contractive and continuous map, where $\mathbf{1}$ is the constant function equal to $1$ for every pair of elements of $X$.
}
\end{remark}
\noindent
It is also interesting to explain consider special cases of condition (\ref{complique}).
\begin{remark}\label{endremark}
\textup{
If $(X,d_\om,a)$ is an $a$-upward $\om$-metric space where $\om$ satisfies conditions $(E_1)$ and $(E_2)$ of Theorem \ref{metricequiv} with $\lambda:[a,\infty) \to [0,\infty)$, then the family $\Psi$ consists of nondecreasing functions $\psi:[a,\infty) \to [a,\infty)$ such that 
\begin{equation}
\displaystyle \sum_{n=1}^\infty \lambda(\psi^{(n)}(\epsilon))<\infty ~~\mbox{for all $\epsilon>a$,}
\end{equation}
since the convergence of the series $\displaystyle \sum_{n=1}^\infty \lambda(\psi^{(n)}(\epsilon))$ is equivalent (by the Cauchy criterion) to $\displaystyle\lim_{n,i \to \infty}  \sum_{j=n}^{n+i} \lambda(\psi^{(j)}(\epsilon))=0$ or $\displaystyle\lim_{n,i \to \infty}  \lambda^{-1}\left(\sum_{j=n}^{n+i} \lambda(\psi^{(j)}(\epsilon))\right)=a$, which is condition (\ref{complique}) when $\om$ is defined by \begin{equation}
\om(u,v)=\lambda^{-1}(\lambda(u)+\lambda(v)) ~~\forall u,v \geq a.
\end{equation}
}
\end{remark}
\noindent
Given a function $\alpha:X \times X \to [0, \infty)$, a map $T:X \to X$ is said to be \textbf{$\alpha$-admissible} (see \cite{samet12}) if for any $x, y\in X$, 
\begin{equation}
\alpha(x,y) \geq 1 \implies \alpha(Tx,Ty) \geq 1.
\end{equation}
The following theorem holds for $\alpha$-admissible, $\alpha$-$\psi$ contractive maps, and easily generalizes the Banach contraction principle.

\begin{theorem}\label{escape}
Let $(X,d_\om,a)$ be a complete O-metric space with $\om$ nondecreasing in both variables, continuous at $(a,a)$, and $\om(a,a)=a$. Let $T: X \to X$ be an $\alpha$-$\psi$ contractive mapping, with $\psi \in \Psi$ and $\alpha:X \times X \to [0,\infty)$ a given map, such that $T$ is $\alpha$-admissible, and $\alpha(x_0,Tx_0) \geq 1$ for some $x_0 \in X$. If either:
\begin{itemize}
\item[(i)] $T$ is sequentially continuous or
\item[(ii)] for any sequence $\{x_n\} \subset X$ such that $\alpha(x_n,x_{n+1}) \geq 1$ for all $n \geq 0$, $x_n \xrightarrow{\text{O}} x$ implies that $\alpha(x_n,x) \geq 1$ for all $n$ sufficiently large,
\end{itemize}
then $T$ has a fixed point. 
\end{theorem}

\begin{proof}
Defining the sequence $\{x_n\}_{n \geq 0}$ by $x_{n+1}=Tx_n$ for all $n \geq 0$, with $x_0$ such that $\alpha(x_0,Tx_0) \geq 1$. By induction and $\alpha$-admissibility of $T$, $\alpha(x_n,x_{n+1}) \geq 1$ for all $n \geq 0$. Since $T$ is $\alpha-\psi$ contractive, then  for all $n \in \mathbb{N}$,
$$d_\om(x_n,x_{n+1}) = d_\om(Tx_{n-1},Tx_n) \leq \alpha(x_{n-1},x_n)d_\om(Tx_{n-1},Tx_n) \leq \psi(d_\om(x_{n-1},x_n)).$$ By induction, $d_\om(x_n,x_{n+1}) \leq \psi^{(n)}(d_\om(x_0,x_1))$, for all $n \in \mathbb{N}$. \\
If $x_0 =x_1$, then $x_0=Tx_0$ is a fixed point of $T$. Suppose now that $x_0 \neq x_1$. \\
Since $d_\om(x_0,x_1) >a$, from (\ref{complique}), $\lim_{n,i \to \infty} \somj{n}{n+i} \psi^{(j)}(d_\om(x_0,x_1))=a$. Fix $\epsilon>a$ and let $N \in \mathbb{N}$ be such that $ \somj{N}{\infty} \psi^{(j)}(d(x_0,x_1)) <\epsilon$. Let $n,m \in \mathbb{N}$ with $m>n>N$. 
From the polygon $\om$-inequality, $d_\om(x_n,x_m) \leq \somj{n}{m-1} \psi^{(j)}(d_\om(x_0,x_1)) <\epsilon$. Thus $\{x_n\}$ is a Cauchy sequence; it converges to some $x^*$ from the completeness from $X$.
\\
If $T$ is sequentially continuous, then $x^*=\lim_{n \to \infty}x_{n+1}=\lim_{n \to \infty} Tx_n = Tx^*$ is a fixed point of $T$. Suppose now that (ii) holds. Then  $\alpha(x_n,x^*) \geq 1$ for $n$ sufficiently large, and from the triangle $\om$-inequality,
\begin{equation*}
\begin{array}{lcl}
d_\om(Tx^*,x^*) &\leq& \om(d_\om(Tx^*,Tx_n),d_\om(Tx_n,x^*))  
\\
&\leq& \om(\alpha(x_n,x^*)d_\om(Tx_n,Tx^*),d_\om(x_{n+1},x^*))
\\
\end{array}
\end{equation*}
By taking the limit, $d_\om(Tx^*,x^*)=a$ hence $x^*=Tx^*$ is a fixed point of $T$.
\end{proof}
\noindent
By virtue of Remark \ref{escapee}, Theorem \ref{most} is a corollary of Theorem \ref{escape}, since any $k$-$\varphi$ contraction is a continuous, $\mathbf{1}$-$\varphi(k,\cdot)$ contraction (see Remark \ref{sleep}). Under conditions $(E_1)$ and $(E_2)$, we obtain the following result:

\begin{corollary}\label{sleep1}
Let $(X,d_\om,a)$ be a complete O-metric space, with $\om$ satisfying conditions $(E_1)$ and $(E_2)$ of Theorem \ref{metricequiv} with $\lambda:[a,\infty) \to [0,\infty)$. Let $\alpha:X \times X \to [0,\infty)$  and $T: X \to X$ be maps such that $T$ is $\alpha$-admissible, $\alpha(x_0,Tx_0) \geq 1$ for some $x_0 \in X$, and $\alpha(x,y)d_\om(Tx,Ty) \leq \psi(d_\om(x,y))$ for all $x,y \in X$, where $\psi:[a,\infty) \to [a,\infty)$ is an nondecreasing function such that $\displaystyle \sum_{n=1}^\infty \lambda(\varphi^{(n)}(\epsilon))<\infty$ for all $\epsilon>a$.
If either:
\begin{itemize}
\item[(i)] $T$ is continuous or
\item[(ii)] for any sequence $\{x_n\} \subset X$ such that $\alpha(x_n,x_{n+1}) \geq 1$ for all $n \geq 0$, $x_n \to x$ implies that $\alpha(x_n,x) \geq 1$ for all $n$ sufficiently large,
\end{itemize}
then $T$ has a fixed point.
\end{corollary}
\begin{proof}
Given Remark \ref{endremark}, the conditions of Theorem \ref{escape} are satisfied hence $T$ has a fixed point.
\end{proof}
\noindent 
If $\lambda$ in Corollary \ref{sleep1} is such that $\lambda(t)=t$ for all $t \geq 0$, the main results in \cite{samet12} are obtained. 
\\
\\
The following corollary holds in the context of b-metric spaces:

\begin{corollary}\label{escapade}
Let $(X,d,s)$ be a complete b-metric space, $\alpha:X \times X \to [0,\infty)$ a mapping, and $T: X \to X$ a map that is $\alpha$-admissible and such that $\alpha(x_0,Tx_0) \geq 1$ for some $x_0 \in X$, and for all $x,y \in X$,
\begin{equation}\label{olaleru}
\alpha d(Tx,Ty) \leq \psi(d(x,y)),
\end{equation}
for a non-decreasing function $\psi:[0,\infty) \to [0,\infty)$ satisfying:
\begin{equation}\label{sarah}
\lim_{n \to \infty} \frac{1}{s^{n-1}} \sum_{j=n}^\infty s^j \psi^{(j)}(\epsilon)=0 ~~\forall \epsilon>0.
\end{equation}
Suppose any of the followig holds:
\begin{itemize}
\item[(i)] $T$ is sequentially continuous or
\item[(ii)] for any sequence $\{x_n\} \subset X$ such that $\alpha(x_n,x_{n+1}) \geq 1$ for all $n \geq 0$, $x_n \to x$ implies that $\alpha(x_n,x) \geq 1$ for all $n$ sufficiently large.
\end{itemize}
Then $T$ has a fixed point. 
\end{corollary}

\begin{proof}
A complete b-metric space $(X,d,s)$ is a complete $\om$-metric space $(X,d_\om,a)$, with $a=0$, $\om$ such that $\om(u,v)=s(u+v)$ for all $u, v \in [0,\infty)$, and $d_\om$ such that $d_\om(x,y)=d(x,y)$ for all $x, y \in X$. The function $\om$ so defined is non-decreasing in both variables, continuous at $(0,0)$, and such that $\om(0,0)=0$.
\\
Let $\varphi:[0,\infty) \to [0,\infty)$ be the nondecreasing function such that (\ref{sarah}) holds. For $n,i \in \mathbb{N}$ and $\epsilon>0$, the $\om$-series $\somj{n}{n+i} \varphi^{(j)} (\epsilon)$ following the pattern of integers $\{1\}_{n \in \mathbb{N}}$ can be computed using equation (\ref{zn}):
\begin{equation}
\begin{array}{lcl}
\somj{n}{n+i} \psi^{(j)} (\epsilon) &=& \somj{0}{i} \psi^{(n+j)} (\epsilon)
\\
&=& \displaystyle \sum_{j=1}^i  s^j \psi^{(n+j-1)}(\epsilon) + s^i \psi^{(n+i)}(\epsilon)
\\
&\leq& \displaystyle s\left( \sum_{j=1}^{i}  s^{j-1} \psi^{(n+j-1)}(\epsilon) + s^i \psi^{(n+i)}(\epsilon) \right)
\\
&=& \displaystyle s \sum_{j=n}^{n+i}  s^{j-n} \psi^{(j)}(\epsilon) 
\\
&=& \displaystyle \frac{1}{s^{n-1}} \sum_{j=n}^{n+i} s^j \psi^{(j)}(\epsilon)
\\
&\leq & \displaystyle \frac{1}{s^{n-1}} \sum_{j=n}^\infty s^j \psi^{(j)}(\epsilon)
\end{array}
\end{equation} 
Given equation (\ref{sarah}), as $n,i \to \infty$, $\displaystyle\lim_{n,i \to \infty} \somj{n}{n+i} \psi^{(j)} (\epsilon) \leq \lim_{n \to \infty} \frac{1}{s^{n-1}} \sum_{j=n}^\infty s^j \psi^{(j)}(\epsilon) =0$. Thus $\displaystyle\lim_{n,i \to \infty} \somj{n}{n+i} \psi^{(j)} (\epsilon)=0$ for all $\epsilon>0$. Equation (\ref{complique}) is satisfied hence $\psi \in \Psi$, and from (\ref{olaleru}), $T:X \to X$ is an $\alpha$-$\psi$ contractive mapping in the sense of Definition \ref{onlyyou}.  All the conditions of Theorem \ref{escape} are therefore satisfied, hence $T$ has a fixed point.
\end{proof} 
\noindent
Corollary \ref{escapade} still holds when condition (\ref{sarah}) is replaced with:
\begin{equation}\label{sarae}
\lim_{n,i \to \infty} s^{\ceil*{\log_2(i+1)}} \sum_{j=n}^{n+i} \psi^{(j)}(\epsilon)=0 ~~\forall \epsilon>0.
\end{equation}
Indeed, if $\epsilon>0$, and $n, i \in \mathbb{N}$, then from (\ref{wn}), the $\om$-series $\somj{n}{n+i} \psi^{(j)}(\epsilon)$ following the pattern of integers $\{2^{\ceil*{\log_2 n}-1} \}_{n \in \mathbb{N}}$ is such that:
\begin{equation}
\somj{n}{n+i} \psi^{(j)}(\epsilon) \leq s^{\ceil*{\log_2(i+1)}}  \sum_{j=n}^{n+i} \psi^{(j)}(\epsilon).
\end{equation}


\section*{Declarations}
\subsection*{Availability of data and materials} 
Not applicable.

\subsection*{Competing interests}
The authors declare that they have no competing interests.


\subsection*{Funding} 
None.

\subsection*{Authors' Contributions} 
All authors contributed meaningfully to this research work, and read and approved its final manuscript.

\subsection*{Acknowledgments}
None.

\end{document}